\documentclass[journal]{IEEEtran}

\usepackage{etex}
\usepackage[vlined,ruled]{algorithm2e}
\usepackage[dvipsnames]{xcolor}
\usepackage{pgfpages}
\usepackage[cmex10]{amsmath}		
\usepackage{amssymb, mathrsfs}
\usepackage{cite}
\usepackage{url}

\usepackage{dsfont}
\usepackage{siunitx}
\usepackage{verbatim}									
\usepackage{mathtools}									
\usepackage{siunitx}									
\usepackage{mdframed}

\usepackage{graphicx}
\usepackage{epstopdf}

\usepackage[font=small]{caption}

\usepackage{array}
\usepackage{multirow}
\usepackage{enumerate}
\usepackage{booktabs}

\graphicspath{{./fig/}}

\newcommand*{\strongconvex}{2}
\newcommand*{\feedforward}{0.9}

\usepackage{pgfplots}
\pgfplotsset{width=8cm,compat=newest}
\newcommand*{\damping}{0.006}%
\newcommand*{\freq}{25}%
\pgfmathsetmacro{\freqd}{sqrt(1-(\damping)^2)*\freq}%

\pgfplotsset{
    standard/.style={
    axis x line=middle,
    axis y line=middle,
    enlarge x limits=0.15,
	enlarge y limits=0.15,
	every axis plot post/.style={mark options={fill=black}},
	}
}

\pgfplotsset{%
    ,compat=1.12
    ,every axis x label/.style={at={(current axis.right of origin)},anchor=north west}
    ,every axis y label/.style={at={(current axis.above origin)},anchor=north east}
    }

\usepackage{tikz}
\usetikzlibrary{arrows}
\usetikzlibrary{decorations.pathmorphing}
\usetikzlibrary{decorations.markings}
\usetikzlibrary{snakes}
\usetikzlibrary{matrix}
\usetikzlibrary{calc}
\usetikzlibrary{patterns}
\usetikzlibrary{positioning}
\usetikzlibrary{shapes}
\usetikzlibrary{shapes.symbols}
\usetikzlibrary{shapes.geometric}
\usetikzlibrary{shadows.blur}
\usetikzlibrary{shapes.arrows}
\usetikzlibrary{shapes.multipart}
\usetikzlibrary{shapes.callouts}
\usetikzlibrary{shapes.misc}

\tikzstyle{every node}=[font=\small]
\tikzstyle{every path}=[line width=0.8pt,line cap=round,line join=round]

\usepackage[american,smartlabels]{circuitikz}
\ctikzset{bipoles/thickness=1}
\ctikzset{bipoles/length=0.8cm}
\ctikzset{bipoles/diode/height=.375}
\ctikzset{bipoles/diode/width=.3}
\ctikzset{tripoles/thyristor/height=.8}
\ctikzset{tripoles/thyristor/width=1}
\ctikzset{bipoles/vsourceam/height/.initial=.7}
\ctikzset{bipoles/vsourceam/width/.initial=.7}

\newcommand{\real}{\mathbb{R}}

\newcommand{\integer}{\mathbb{Z}}

\newcommand{\setdef}[2]{\{#1 \;|\; #2\}}

\newcommand{\vect}[1]{\mathbbold{#1}}

\newcommand{\vzeros}[1][]{\vect{0}_{#1}}
\DeclareSymbolFont{bbold}{U}{bbold}{m}{n}
\DeclareSymbolFontAlphabet{\mathbbold}{bbold}

\newcommand{\map}[3]{#1: #2 \rightarrow #3}

\renewcommand{\theenumi}{(\roman{enumi})}

\newcommand{\tb}{\color{blue}}

\newcommand\oprocendsymbol{\hbox{$\square$}}
\newcommand\oprocend{\relax\ifmmode\else\unskip\hfill\fi\oprocendsymbol}

\newtheorem{theorem}{Theorem}[section]
\newtheorem{lemma}[theorem]{Lemma}
\newtheorem{definition}[theorem]{Definition}
\newtheorem{corollary}[theorem]{Corollary}
\newtheorem{proposition}[theorem]{Proposition}

\newtheorem{remark}[theorem]{Remark}

\newtheorem{example}[theorem]{Example}

\newenvironment{pfof}[1]{\vspace{1ex}\noindent{\itshape Proof of
    #1:}\hspace{0.5em}} {\hfill\oprocend\vspace{1ex}}
\newenvironment{proof}[1]{\vspace{1ex}\noindent{\itshape Proof:}\hspace{0.5em}} {\hfill\oprocend\vspace{1ex}}

\renewcommand{\tb}{\color{black}}

\begin{document}
\title{Equilibrium-Independent Dissipativity with Quadratic Supply Rates}

\author{John~W.~Simpson-Porco,~\IEEEmembership{Member,~IEEE}\\
\thanks{J.~W.~Simpson-Porco is with the Department of Electrical and Computer Engineering, University of Waterloo, Waterloo ON, N2L 3G1 Canada. Correspondence: {\texttt{jwsimpson@uwaterloo.ca}}.}
\thanks{This work was supported in part by the NSERC Discovery Grant  RGPIN-2017-04008 and by University of Waterloo start-up funding.}
}

\markboth{Submitted to IEEE Transactions on Automatic Control. This version: \today}%
{Submitted to IEEE Transactions on Automatic Control. This version: \today}

\maketitle

\begin{abstract}
Equilibrium-independent dissipativity (EID) is a recently introduced system property which requires a system to be dissipative with respect to any forced equilibrium configuration. This paper is a detailed examination of EID with quadratic supply rates for a common class of nonlinear control-affine systems. {\tb We provide an algebraic characterization of EID for such systems} in the spirit of the Hill-Moylan lemma, where the usual stability condition is replaced by an incremental stability condition. Based on this characterization, we state results concerning internal stability, feedback stability, and absolute stability of EID systems. Finally, {\tb we study EID for discrete-time systems}, providing the relevant definitions and {\tb an} analogous Hill-Moylan-type {\tb characterization}. Results for both continuous-time and discrete-time systems are illustrated through examples on physical systems and convex optimization algorithms.
\end{abstract}



\section{Introduction}
\label{Section: Introduction}



Dissipation inequalities provide a general framework for the analysis and design of interconnected nonlinear dynamical systems. Introduced by Williems in \cite{JCW:72a}, dissipativity is an input-output system property which unifies classical properties such as finite-gain, passivity, and conicity \cite{GZ:66}. 
Further advances in \cite{PM:74,HM:76} by Hill and Moylan characterized dissipativeness for control-affine systems in terms of a system of nonlinear equations. 
Dissipative systems theory and associated control design techniques are now fairly mature, with several  reference books available \cite{RS-MJ-PK:97,BB-RL-BM-OE:07,AJvdS:16}.

When applied to state-space systems for the purposes of stability analysis, dissipation inequalities are referenced to a chosen equilibrium input-state-output configuration $(\bar{u},\bar{x},\bar{y})$, {\tb which is typically} taken to be the origin. If several such dissipative systems are interconnected with one another, the origin is an equilibrium point for the closed-loop system, and dissipativity theory provides tools for assessing its stability \cite{AJvdS:16}.
This framework however assumes considerable knowledge of the equilibrium sets of the individual subsystems, and this may not be justified in applications. When considering uncertain, large-scale, nonlinear systems, equilibrium sets of subsystems may be uncertain or otherwise difficult to characterize. Further complicating the situation, the very act of interconnection between subsystems will induce a closed-loop equilibrium set, determined by the simultaneous solution of all subsystem equilibrium equations and all interconnection constraints. When many uncertain systems are interconnected, explicitly calculating this equilibrium set may prove infeasible. It then becomes challenging to construct classical storage functions for the subsystems in order to verify internal stability and/or I/O properties of the interconnection; classical dissipativity falls short as an effective tool.

One remedy to these issues is termed \emph{incremental dissipativity}, which requires that a dissipation inequality hold along any two arbitrary trajectories of a forced system \cite{AP-LM:08}. A closely related property termed differential dissipativity is discussed in \cite{FF-RS:13,AJvdS:13}. Under appropriate technical assumptions, {\tb incremental dissipativity} implies the existence of a {\tb unique equilibrium trajectory towards which all other trajectories converge}. As such, {\tb incremental dissipativity} has proven useful for studying output regulation \cite{AP-LM:08,MB-CDP:15} and synchronization of interconnected systems \cite{GBS-RS:07,TL-DJH-JZ:14,HP-JZ:16}, where all subsystem trajectories converge to a common global steady-state trajectory. Incremental dissipativity however is quite demanding as a system property, since often we wish only to establish stability/dissipativity of trajectories with respect to the set of equilibrium configurations, and not with respect to all other possible trajectories.


As an intermediate property between classical and incremental dissipativity, \emph{equilibrium-independent dissipativity} (EID) has recently been introduced \cite{GHH-MA-AKP:11,MB-DZ-FA:13,MB-DZ-FA:14}, requiring a dissipation inequality to hold between any system trajectory and any forced equilibrium point. 
{\tb The utility of this property is that as the operating point of the system moves --- either intentionally due to set-point changes, or unintentionally due to disturbances --- one is guaranteed that the dissipation inequality under consideration will continue to hold with respect to the \emph{new} operating point.}
This property has been used for the control of port-Hamiltonian systems \cite{BJ-RO-EGC-FC:07,SA-RO-RC:15}, for performance certification of interconnected systems \cite{CM-LL-MA-AKP:15,MA-CM-AP:16}, for network congestion control \cite{JTW-MA:04}, for stability analysis of various power system models \cite{ST-MB-CDP:16,ST-CDP:17,CDP-NM:18}, and for analysis of optimization algorithms \cite{JWSP:16f}. Particularly relevant to this paper is \cite{BJ-RO-EGC-FC:07}, where a Lyapunov construction based on the \emph{Bregman divergence} was used to establish equilibrium-independent passivity.\footnote{\tb The use of the Bregman divergence in the control literature apparently traces to \cite{AAA-BEY:01}; we thank N. Monshizadeh for this observation.}

The theory of EID systems presented in \cite{GHH-MA-AKP:11,MB-DZ-FA:13,MB-DZ-FA:14} has not however been developed to the level of the classical dissipativity literature \cite{RS-MJ-PK:97,BB-RL-BM-OE:07,AJvdS:16}, and no consistent, comprehensive reference is available. In addition, two particularly important items absent from the literature are an {\tb algebraic Hill-Moylan-type characterization} of EID, and analogous definitions and results for the discrete-time case. The former is theoretical bedrock and a key step towards EID control design \cite{AJvdS:16}, while the author sees the latter as important for analyzing and designing interconnections of physical systems with optimization algorithms. Putting these future directions/applications to the side, here we focus instead on developing and illustrating the basic theory of equilibrium-independent dissipativity.


\smallskip

\subsection{Contributions}

The overarching goal of this paper is to provide a detailed treatment of EID systems with quadratic supply rates, developing basic characterizations and stability results, and illustrating the results with examples. {\tb We restrict our discussion to} nonlinear control-affine systems with constant input and throughput matrices, in both continuous and discrete-time.\footnote{We consider this particular subclass of control-affine systems because (i) it is sufficient for the applications we have considered, and (ii) it permits a relatively intuitive extension of Hill-Moylan conditions for classical dissipativity to EID.}

There are three main contributions.\footnote{A short version of this paper has been submitted to ACC 2018. The ACC version contains Lemma \ref{Lem:HillMoylanIncremental}, its proof, and the statement of Theorem \ref{Thm:Absolute}. The ACC version does not contain Example \ref{Example:Port}, Example \ref{Example:Grad}, Example \ref{Example:AHU}, Lemma \ref{Lem:InternalStability}, Theorem \ref{Thm:EIDFeedback}, the proof of Theorem \ref{Thm:Absolute}, Example \ref{Ex:SMIB}, Lemma \ref{Lem:MaxMonotone}, and Lemma \ref{Lem:Intersection}, or any of the material from Section \ref{Sec:DTEID}. This is noted to emphasize that the contributions of this paper differ substantially from the conference version.} First, we show in Section \ref{Sec:EID} that EID can be characterized in terms of an appropriately modified Hill-Moylan lemma \cite{HM:76}. The key modification is that the usual stability-like condition is replaced by an incremental stability-like condition. Roughly speaking, the results can be interpreted as saying that dissipativity plus an appropriate incremental stability-like condition yields EID; we present various examples illustrating the results. Second, in Section \ref{Sec:Stability}, we study stability of EID systems, stating results for internal and feedback stability, and study an equilibrium-independent variant of the absolute stability problem. Third and finally, in Section \ref{Sec:DTEID} we consider the discrete-time case, providing the relevant definitions, corresponding Hill-Moylan-type conditions, and illustrating how the results can be applied to analyze the gradient method for convex optimization. 

Two major {\tb implications} of our results are that (i) EID can be {\tb established and applied to problems} in much the same way as standard dissipativity, and (ii) for square EID systems in feedback, the existence/uniqueness of closed-loop equilibria can be inferred by studying the monotonicity of the subsystem I/O relations.

\subsection{Notation}

\smallskip

The set $\real$ (resp. $\real_{\geq 0}$) is the set of real (resp. nonnegative) numbers. The $n \times n$ identity matrix is $I_n$, $\vzeros[]$ is a matrix of zeros of appropriate dimension, while $\vzeros[n]$ is the $n$-vector of all zeros. Throughout, $\|x\|_2 = (x^{\sf T}x)^{1/2}$ denotes the 2-norm of $x$, while for $P = P^{\sf T} \succ \vzeros[]$, $\|x\|_P = (x^{\sf T}Px)^{1/2}$; when convenient, we will (ab)use this notation even if $P \succeq \vzeros[]$. {\tb The set of real-valued square-integrable signals $\map{v}{[0,\infty)}{\real^m}$ is denoted by $\mathscr{L}_{2}^m[0,\infty)$, with $\mathscr{L}_{2e}^m[0,\infty)$ denoting the associated extended signal space \cite[Chapter 1]{AJvdS:16}; the corresponding discrete-time spaces are denoted by $\ell_{2}^{m}[0,\infty)$ and $\ell_{2e}^m[0,\infty)$.} For a twice-differentiable function $\map{V}{\real^n}{\real}$, $\map{\nabla V}{\real^n}{\real^n}$ is its gradient while $\map{\nabla^2 V}{\real^n}{\real^{n \times n}}$ is its Hessian. {\tb A differentiable function $\map{V}{\real^n}{\real}$ is \emph{convex} if
$$
[\nabla V(x)-\nabla V(z)]^{\sf T}(x-z) \geq k(x,z)\|x-z\|_2^2
$$
for all $x,z \in \real^n$ and some function $\map{k}{\real^n \times \real^n}{\real_{\geq 0}}$. If $k(x,z) > 0$ for all $x \neq z$, then $V$ is \emph{strictly} convex, and if $k(x,z) \geq \mu > 0$ for all $(x,z)$, then $V$ is $\mu$-\emph{strongly} convex; in the twice differentiable case, these statements are equivalent to $\nabla^2V(x) \succ \vzeros[]$ and $\nabla^2 V(x) \succeq \mu I_n$ for all $x \in \real^n$, respectively.
}

\section{Nonlinear Dissipative Systems}
\label{Sec:ReviewDissipative}

\subsection{Control-Affine Systems and Forced Equilibria}
\label{Sec:ControlAffine}

Consider the continuous-time nonlinear control-affine system {\tb with constant input and throughput matrices}
\begin{equation}\label{Eq:NonlinearSystem}
\Sigma:\,\begin{cases}
\begin{aligned}
\dot{x}(t) &= f(x(t)) + Gu(t)\\
y(t) &= h(x(t)) + Ju(t)
\end{aligned}
\end{cases}
\end{equation}
with state {\tb $x(t) \in \mathcal{X} := \real^n$, input $u(t) \in \mathcal{U} := \real^m$ and output $y(t) \in \mathcal{Y} := \real^p$ where $m, p \leq n$.  The maps $\map{f}{\mathcal{X}}{\real^n}$ and $\map{h}{\mathcal{X}}{\mathcal{Y}}$ are assumed to be sufficiently smooth such that trajectories are forward complete  for all initial conditions $x(0) \in \mathcal{X}$ and all input functions $u(\cdot) \in \mathscr{L}_{2e}^m[0,\infty)$, with corresponding output trajectories $y(\cdot) \in \mathscr{L}_{2e}^{p}[0,\infty)$.} The input matrix $G \in \real^{n \times m}$ is constant and has rank $m$ (full column rank). The throughput matrix $J \in \real^{p \times m}$ is constant. {\tb An \emph{equilibrium configuration} of \eqref{Eq:NonlinearSystem} is a triple $(\bar{u},\bar{x},\bar{y}) \in \mathcal{U} \times \mathcal{X} \times \mathcal{Y}$ satisfying }
\begin{equation}\label{Eq:NonlinearEquilibrium}
\begin{aligned}
\vzeros[n] &= f(\bar{x}) + G\bar{u}\\
\bar{y} &= h(\bar{x}) + J\bar{u}\,.
\end{aligned}
\end{equation}
When $m = n$, the system is fully actuated and for any desired equilibrium point $\bar{x} \in \mathcal{X}$, $\bar{u} = -G^{-1}f(\bar{x})$ is the associated {\tb equilibrium} input. When $m < n$, let $G^{\perp} \in \real^{(n-m)\times n}$ be a full-rank left annihilator of $G$; {\tb that is, $G^{\perp}$ satisfies} $G^{\perp}G = \vzeros[]$ and $\mathrm{rank}({G^{\perp}}) = n-m$ \cite[Lemma 2]{RO-AVDS-FC-AA:08}. It follows that
$$
\mathcal{E}_{\Sigma} :=
\begin{cases}
\mathcal{X} & \text{if} \quad m = n\\
\setdef{\bar{x} \in \mathcal{X}}{G^{\perp}f(\bar{x}) = \vzeros[n-m]} & \text{if} \quad m < n
\end{cases}
$$ 
is the set of assignable equilibrium points. For every $\bar{x} \in \mathcal{E}_{\Sigma}$, we have the associated unique {\tb equilibrium} input and output
\begin{equation}\label{Eq:InputOutputMappings}
\begin{aligned}
\bar{u} &= k_{u}(\bar{x}) := -(G^{\sf T}G)^{-1}G^{\sf T}f(\bar{x})\\
\bar{y} &= k_{y}(\bar{x}) := h(\bar{x}) - J(G^{\sf T}G)^{-1}G^{\sf T}f(\bar{x})\,.
\end{aligned}
\end{equation}
{\tb While the input-to-state map $\map{k_u}{\mathcal{X}}{\mathcal{U}}$ defined above} is a function, it is useful to {\tb reinterpret} it as a {\tb \emph{relation}
$$
\mathcal{K}_u := \setdef{(x,u)}{u+(G^{\sf T}G)^{-1}G^{\sf T}f(x) = \vzeros[m]} \subset \mathcal{X} \times \mathcal{U}\,,
$$
and consider the inverse relation $\mathcal{K}_u^{-1} \subset \mathcal{U} \times \mathcal{X}$, which relates the domain of equilibrium inputs to the codomain of forced equilibria. As is standard, we overload the notation and interpret the relation $\mathcal{K}_{u}^{-1}(\cdot)$ as a set-valued mapping when convenient. From \eqref{Eq:InputOutputMappings} then, we may define an equilibrium input/output (I/O) {relation} $\mathcal{K}_{\Sigma} := k_y \circ \mathcal{K}_u^{-1} \subseteq \mathcal{U} \times \mathcal{Y}$, or equivalently
$$
\mathcal{K}_{\Sigma} := \setdef{(\bar{u},\bar{y}) \in \mathcal{U}\times\mathcal{Y}}{\text{there exists}\,\,\bar{x}\in\mathcal{X}\,\,\text{solving}\,\,\eqref{Eq:InputOutputMappings}}\,,
$$
which relates compatible steady-state inputs and outputs.}

\smallskip

\begin{remark}{\bf (Assignable Equilibria):}\label{Rem:LTI}
If $\vzeros[n-m]$ is a regular value for $E(x) := G^{\perp}f(x)$, then the associated fiber $\mathcal{E}_{\Sigma} = E^{-1}(\vzeros[n-m])$ is a $m$-dimension embedded submanifold of $\real^n$ \cite[
Corollary 5.24]{JML:03}. In the case of LTI systems where $f(x) = Fx$ with $F \in \real^{n \times n}$, the set of assignable equilibria becomes $\mathcal{E}_{\Sigma} = \ker(G^{\perp}F)$. If in addition $F$ is invertible, this simplifies further to $\mathcal{E}_{\Sigma} = F^{-1}\mathrm{Im}(G)$. \hfill \oprocend
\end{remark}

\smallskip


\subsection{Classical Dissipativity of Control-Affine Systems}
\label{Sec:ClassicDissipativityCT}

We provide a brief review of dissipativity theory for control-affine nonlinear systems; see \cite{RS-MJ-PK:97,BB-RL-BM-OE:07,AJvdS:16} for various overviews of dissipativity and related concepts. {\tb In this subsection, we make the additional} assumptions for \eqref{Eq:NonlinearSystem} that $f(\vzeros[n]) = \vzeros[n]$ and $h(\vzeros[n]) = \vzeros[p]$, so that $(\bar{u},\bar{x},\bar{y}) = (\vzeros[m],\vzeros[n],\vzeros[p])$ is an equilibrium configuration. Let $\map{{\sf w}}{\mathcal{U} \times \mathcal{Y}}{\real}$ be a continuous function called the \emph{supply rate}. The system $\Sigma$ in \eqref{Eq:NonlinearSystem} is \emph{dissipative} with respect to the supply rate ${\sf w}(u,y)$ if there exists a continuously differentiable \emph{storage function} $\map{V}{\mathcal{X}}{\real_{\geq 0}}$ with $V(\vzeros[n]) = 0$ such that
\begin{equation}\label{Eq:Dissipation}
\frac{\mathrm{d}}{\mathrm{d}t} V(x(t)) := \nabla V(x)^{\sf T}(f(x)+Gu) \leq {\sf w}(u(t),y(t))
\end{equation}
{\tb for all $t \geq 0$ and all measurable inputs $u(\cdot) \in \mathscr{L}_{2e}^{m}[0,\infty)$.} The inequality \eqref{Eq:Dissipation} is called a \emph{dissipation inequality}; the interpretation is that the rate of change of energy $V(x)$ {\tb stored} by the system is less than the supplied power ${\sf w}(u,y)$. {\tb In this paper we} focus exclusively on quadratic supply rates
\begin{equation}\label{Eq:SupplyNormal}
{\sf w}(u,y) = \begin{bmatrix}
y \\ u
\end{bmatrix}^{\sf T}\begin{bmatrix}
Q & S\\
S^{\sf T} & R
\end{bmatrix}\begin{bmatrix}
y \\ u
\end{bmatrix}\,,
\end{equation}
where $Q = Q^{\sf T}, S$, and $R = R^{\sf T}$ are matrices of appropriate dimensions. {\tb To ensure that the inequality \eqref{Eq:Dissipation} is not trivially satisfied, we make the standard assumption that the block matrix in \eqref{Eq:SupplyNormal} is sign-indefinite \cite{HM:76}.}
The supply rate \eqref{Eq:SupplyNormal} contains some common I/O system properties as special cases, including passivity $(Q,S,R) = (\vzeros[],\frac{1}{2}I_m,\vzeros[])$ and finite $\mathscr{L}_2$-gain $(Q,S,R) = (-I_p, \vzeros[], \gamma^2 I_m)$ for $\gamma \geq 0$.
%
%
%
%
%
The key characterization of quadratically dissipative continuous-time control-affine systems is due to Hill and Moylan.

\smallskip
\begin{lemma}{\bf (Hill-Moylan Conditions \cite{HM:76}):}\label{Lem:HillMoylan}
The control-affine system $\Sigma$ in \eqref{Eq:NonlinearSystem} is dissipative with respect to the supply rate \eqref{Eq:SupplyNormal} with continuously-differentiable storage function $\map{V}{\mathcal{X}}{\real_{\geq 0}}$ if and only if there exists an integer $k > 0$, a matrix $W \in \real^{k \times m}$ and a function $\map{l}{\mathcal{X}}{\real^k}$ such that
\begin{subequations}\label{Eq:HillMoylan}
\begin{align}
\label{Eq:HillMoylan1}
\nabla V(x)^{\sf T}f(x) &= h(x)^{\sf T}Qh(x) - l(x)^{\sf T}l(x)\\
\label{Eq:HillMoylan2}
\frac{1}{2}\nabla V(x)^{\sf T}G &= h(x)^{\sf T}(QJ+S) - l(x)^{\sf T}W\\
\label{Eq:HillMoylan3}
W^{\sf T}W &= R + J^{\sf T}S + S^{\sf T}J + J^{\sf T}QJ\,.
\end{align}
\end{subequations}
\end{lemma}
\smallskip
%

In most applications, the first equation in \eqref{Eq:HillMoylan} enforces some type of stability; the remaining equations ensure a proper matching of inputs and outputs to generate the supply rate \eqref{Eq:SupplyNormal}. {\tb  When specialized to LTI systems $\dot{x} = Fx + Gu\,, y = Hx + Ju$, with quadratic storage functions $V(x) = x^{\sf T}Px$, $P = P^{\sf T} \succeq \vzeros[]$, Lemma \ref{Lem:HillMoylan} states that dissipativity with respect to the quadratic supply rate \eqref{Eq:SupplyNormal} is equivalent to the existence of an integer $k > 0$ and matrices $L \in \real^{k \times n}, W \in \real^{k \times m}$ solving the linear matrix equality
\begin{equation}\label{Eq:KYP}
\begin{aligned}
\begin{bmatrix}
F^{\sf T}P + PF & PG\\
G^{\sf T}P & \vzeros[]
\end{bmatrix} &- \begin{bmatrix}
H & J\\
\vzeros[] & I_m
\end{bmatrix}^{\sf T}\begin{bmatrix}
Q & S\\
S^{\sf T} & R
\end{bmatrix}\begin{bmatrix}
H & J\\
\vzeros[] & I_m
\end{bmatrix}\\
&+ \begin{bmatrix}
L^{\sf T} \\ W^{\sf T}
\end{bmatrix}\begin{bmatrix}
L & W
\end{bmatrix} = \vzeros[]\,.
\end{aligned}
\end{equation}


}
%

\section{Equilibrium-Independent Dissipativity for Continuous-Time Control-Affine Systems}
\label{Sec:EID}

The presented state-space definitions of dissipativity implicitly reference a specific equilibrium configuration (the origin). Often however, we are interested in operating a control system around an equilibrium configuration $(\bar{u},\bar{x},\bar{y})$, and we wish to establish input/output properties with respect to this forced equilibrium configuration. In general, verifying dissipativeness with respect to the forced equilibrium must be done with a \emph{new} storage candidate $V_{\bar{x}}(x)$, which depends on the equilibrium $\bar{x}$. Simply shifting a storage function $V(x)$ as used in Lemma \ref{Lem:HillMoylan} need not suffice, as the following simple example shows.

\begin{example}{\bf (Second-Order System):}\label{Ex:TwoStateSystem}
Consider the second-order system
$$
\begin{aligned}
\dot{x}_1 &= x_2\,, \quad \dot{x}_2 = -\nabla U(x_1) - x_2 + u\\
y &= x_2
\end{aligned}
$$
where $\map{U}{\real}{\real}$ is {\tb differentiable and} strictly convex, with $\nabla U(0) = 0$. Clearly $(\bar{u},\bar{x}_1,\bar{x}_2,\bar{y}) = (0,0,0,0)$ is an equilibrium configuration, and the storage function $V(x) = \frac{1}{2}x_2^{\sf T}x_2 + U(x_1) - U(0)$ satisfies $V(0,0) = 0$ and certifies output-strict passivity:
$$
\begin{aligned}
\dot{V} = \nabla U(x_1)\cdot x_2 - x_2\cdot\nabla U(x_1) - x_2^2 + x_2u &= -y^2 + yu\\
&:= {\sf w}(u,y)\,.
\end{aligned}
$$
Consider now a forced equilibrium configuration $(\bar{u},\bar{x}_1,0,0)$, where {\tb $\bar{u} = k_u(\bar{x}) := \nabla U(\bar{x}_1)$}. A natural choice for a storage candidate is $V_{\bar{x}}(x) =  V(x) - U(\bar{x}_1)$, satisfying $V_{\bar{x}}(\bar{x}_1,0) = 0$. However, a similar calculation shows that
$$
\begin{aligned}
\dot{V}_{\bar{x}} &= -(y-\bar{y})^2 + (y-\bar{y})u\\
&\neq {\sf w}(u-\bar{u},y-\bar{y})\,,
\end{aligned}
$$
and therefore $V_{\bar{x}}(x)$ does not establish the desired equilibrium-independent passivity property. \hfill \oprocend
\end{example}


\smallskip

The concept of equilibrium-independent dissipativity (EID) requires dissipativity of a system with respect to any equilibrium configuration \cite{GHH-MA-AKP:11,MB-DZ-FA:14,MA-CM-AP:16}. Our definition roughly follows \cite{MB-DZ-FA:14,MA-CM-AP:16}.

\medskip

\begin{definition}{\bf (Equilibrium-Independent Dissipativity):}\label{Def:EID}
The control-affine system \eqref{Eq:NonlinearSystem} is equilibrium-independent dissipative (EID) with supply rate $\map{\sf w}{\mathcal{U}\times\mathcal{Y}}{\real}$ if, for every equilibrium $\bar{x} \in \mathcal{E}_{\Sigma}$, there exists a continuously-differentiable storage function $\map{V_{\bar{x}}}{\mathcal{X}}{\real_{\geq 0}}$ such that $V_{\bar{x}}(\bar{x}) = 0$ and
\begin{equation}\label{Eq:EID}
\frac{\mathrm{d}}{\mathrm{d}t} V_{\bar{x}}(x(t)) := \nabla V_{\bar{x}}(x)^{\sf T}(f(x)+Gu) \leq {\sf w}(u-\bar{u},y-\bar{y})\,,
\end{equation}
{\tb for all $t \geq 0$ and all measurable inputs $u(\cdot) \in \mathscr{L}_{2e}^{m}[0,\infty)$,} where $\bar{u} = k_u(\bar{x})$, $\bar{y} = k_y(\bar{x})$. A set of storage functions $\{V_{\bar{x}}(x)\,,\,\, \bar{x} \in \mathcal{E}_{\Sigma}\}$ satisfying \eqref{Eq:EID} is an \emph{EID storage function family}.
\end{definition}

\smallskip

Note that in Definition \ref{Def:EID}, the {\tb supply rate ${\sf w}(\cdot,\cdot)$ does not} depend on $\bar{x}$. In other words, EID requires uniformity in the supply rate across all assignable equilibrium points.

Suppose that $\tilde{x} \in \mathcal{E}_{\Sigma}$ is another assignable equilibrium point with associated equilibrium inputs/outputs $\tilde{u} = k_u(\tilde{x})$ and $\tilde{y} = k_y(\tilde{x})$. If one selects $(x,u) = (\tilde{x},\tilde{u})$ in Definition \ref{Def:EID}, then the left-hand side of \eqref{Eq:EID} becomes zero and we find that ${\sf w}(\tilde{u}-\bar{u},\tilde{y}-\bar{y}) \geq 0$. One quickly arrives at the following result.

\smallskip


\begin{lemma}{\bf (I/O Relation Constraint):}\label{Lem:IOMappings}
If $\Sigma$ is EID, then its equilibrium I/O relation $\mathcal{K}_{\Sigma}$ satisfies
\begin{equation}\label{Eq:IOMapping}
\begin{bmatrix}
\bar{y}-\tilde{y}\\
\bar{u} - \tilde{u}
\end{bmatrix}^{\sf T}
\begin{bmatrix}
Q & S\\
S^{\sf T} & R
\end{bmatrix}\begin{bmatrix}
\bar{y}-\tilde{y}\\
\bar{u} - \tilde{u}
\end{bmatrix} \geq 0
\end{equation}
for any two pairs $(\bar{u},\bar{y}), (\tilde{u},\tilde{y}) \in \mathcal{K}_{\Sigma}$. 
\end{lemma}

\medskip

In the square case $m=p$ with $(Q,S,R) = (\vzeros[],\frac{1}{2}I_m,\vzeros[])$, the inequality \eqref{Eq:IOMapping} says that $\mathcal{K}_{\Sigma}$ is a \emph{monotone} (i.e., incrementally passive) relation. Monotone relations have been extensively studied in the convex analysis literature, but most useful results require a slightly stronger property termed \emph{maximal} monotonicity \cite[Chap. 20]{HB-PC:11}. Lemma \ref{Lem:MaxMonotone} in the appendix presents some sufficient conditions which ensure that a monotone equilibrium I/O relation is $\mathcal{K}_{\Sigma}$ is maximally monotone.




For nonlinearities $\map{\psi}{\mathcal{D} \subset \real^m}{\real^p}$, all storage functions in Definition \ref{Def:EID} are taken as zero and $\psi$ is EID if
\begin{equation}\label{Eq:EIDNonlinearity}
\begin{bmatrix}
\psi(z_2)-\psi(z_1)\\
z_2 - z_1
\end{bmatrix}^{\sf T}
\begin{bmatrix}
Q & S\\
S^{\sf T} & R
\end{bmatrix}\begin{bmatrix}
\psi(z_2)-\psi(z_1)\\
z_2 - z_1
\end{bmatrix} \geq 0
\end{equation}
for every $z_1,z_2 \in \mathcal{D}$. In the square case where $m = p$, EID encompasses several standard classes of mappings associated with gradients of convex functions \cite{HB-PC:11}, including
\begin{enumerate}
\item \emph{monotone:} $Q = \vzeros[]$, $S = \frac{1}{2}I_m$, $R = \vzeros[]$,
\item \emph{$\nu$-strongly monotone:} $Q = \vzeros[]$, $S = \frac{1}{2}I_m$, $R = -\nu I_m$,
\item \emph{$\rho$-cocoercive:} $Q = -\rho I_m$, $S = \frac{1}{2}I_m$, $R = \vzeros[]$,
\end{enumerate}
as well as \emph{$\gamma$-Lipschitz} mappings with $Q = -I_m$, $S = \vzeros[]$, and $R = \gamma^2 I_m$.

%





\subsection{Hill-Moylan Conditions for EID}
\label{Sec:Main}

Our first major result gives a version of Lemma \ref{Lem:HillMoylan} appropriate for EID systems. The Lyapunov construction is inspired by \cite{BJ-RO-EGC-FC:07}, and provides a convenient parameterization of the EID storage function family $\{V_{\bar{x}}(x), x \in \mathcal{E}_{\Sigma}\}$.

\smallskip

\begin{lemma}{\bf (Hill-Moylan Conditions for EID):}\label{Lem:HillMoylanIncremental}
Consider the control-affine system $\Sigma$ in \eqref{Eq:NonlinearSystem}. Let $\map{V}{\mathcal{X}}{\real_{\geq 0}}$ be  continuously differentiable and convex, and for $\bar{x}\in\mathcal{E}_{\Sigma}$, let
\begin{equation}\label{Eq:Bregman}
V_{\bar{x}}(x) := V(x) - V(\bar{x}) - \nabla V(\bar{x})^{\sf T}(x-\bar{x})\,.
\end{equation}
The system $\Sigma$ is EID with respect to the {\tb quadratic} supply rate ${\sf w}(u,y)$ in \eqref{Eq:SupplyNormal} with storage function family $\{V_{\bar{x}}(x)\,,\,\,\bar{x} \in \mathcal{E}_{\Sigma}\}$ if and only if there exists an integer $k > 0$, a matrix $W \in \real^{k \times m}$, {\tb and a function $\map{\ell}{\mathcal{X}\times\mathcal{X}}{\real^k}$ such that
\begin{subequations}\label{Eq:HillMoylanEID}
\begin{align}
\label{Eq:HillMoylanEID1}
&\begin{aligned}
&\hspace{-2em}\,[\nabla V(x)-\nabla V(\bar{x})]^{\sf T}[f(x)-f(\bar{x})]\\
&\hspace{-2em}  = [h(x)-h(\bar{x})]^{\sf T}Q[h(x)-h(\bar{x})] -\|\ell(x,\bar{x})\|_2^2\\
\end{aligned}\\
\label{Eq:HillMoylanEID2}
&\begin{aligned}
\hspace{-3em}\frac{1}{2}[\nabla V(x)-\nabla V(\bar{x})]^{\sf T}G &= [h(x)-h(\bar{x})]^{\sf T}(QJ+S)\\
&\quad - \ell(x,\bar{x})^{\sf T}W\\
\end{aligned}\\
\label{Eq:HillMoylanEID3}
W^{\sf T}W &= R + J^{\sf T}S + S^{\sf T}J + J^{\sf T}QJ
\end{align}
\end{subequations}
for all $(x,\bar{x}) \in \mathcal{X} \times \mathcal{E}_{\Sigma}$. The function $\ell(x,\bar{x})$ appearing in \eqref{Eq:HillMoylanEID1}--\eqref{Eq:HillMoylanEID2} may always be chosen to have the form 
$$
\ell(x,\bar{x}) = l(x)-l(\bar{x}) + Tq(x,\bar{x})\,,
$$
where $\map{l}{\mathcal{X}}{\real^k}$, the columns of $T \in \real^{k \times r}$ with $r = \dim(\ker(W^{\sf T}))$ form a basis for $\ker(W^{\sf T})$, and $\map{q}{\mathcal{X}\times\mathcal{X}}{\real^r}$ satisfies $q(x,x) = \vzeros[r]$ for all $x \in \mathcal{X}$.
}
\end{lemma}

\smallskip

From a procedural point of view, Lemma \ref{Lem:HillMoylanIncremental} says that if one can find a convex function $V(x)$ along with {\tb $\ell(x,\bar{x})$ and $W$} satisfying \eqref{Eq:HillMoylanEID1}--\eqref{Eq:HillMoylanEID3}, then \eqref{Eq:Bregman} parameterizes the entire EID storage function family certifying EID with quadratic supply rate \eqref{Eq:SupplyNormal}. Here we have opted to state the result in terms of the existence of these quantities for the particular EID storage function $V_{\bar{x}}(x)$ in \eqref{Eq:Bregman}, rather than infer the existence of an EID storage function family from an appropriately defined input/output EID property. This choice conforms with how storage functions are selected in practice, and highlights the utility of the particular parameterization \eqref{Eq:Bregman}.

\smallskip

\begin{remark}{\bf (Incremental Stability):}\label{Rem:IncrementalStability}
The condition \eqref{Eq:HillMoylanEID1} strengthens the standard stability-like condition \eqref{Eq:HillMoylan1}, requiring instead an incremental-stability-like property. To see why this terminology is appropriate, consider the case of a quadratic storage function $V(x) = \frac{1}{2}x^{\sf T}Px$, $P \succ \vzeros[]$, state measurement $h(x) = x$, and $Q \prec \vzeros[]$. The first condition \eqref{Eq:HillMoylanEID1} then implies that
\begin{equation}\label{Eq:IncrementalLypaunov1}
[P(x-\bar{x})]^{\sf T}f(x) + [P(\bar{x}-x)]^{\sf T}f(\bar{x}) \leq -\varepsilon V_{\bar{x}}(x)
\end{equation}
for some $\varepsilon > 0$. If this holds for all $x, \bar{x} \in \real^n$, it follows that $V_{\bar{x}}(x) = \frac{1}{2}\|x-\bar{x}\|_P^2$ is an \emph{incremental Lyapunov function} \cite{MZ-RM:11} for the unforced system $\dot{x} = f(x)$. Alternatively, it can be shown \cite[Appendix A]{AP-AP-NVDW-HN:04} that \eqref{Eq:IncrementalLypaunov1} implies the matrix inequality
$$
\left(\frac{\partial f}{\partial x}(x)\right)^{\sf T}P + P\left(\frac{\partial f}{\partial x}(x)\right) \prec \vzeros[]\,,
$$
for all $x \in \real^n$, which is the Demidovich condition for convergence/incremental stability/contraction \cite{JWSP-FB:12za}.  {\tb We note however that \eqref{Eq:Bregman} is in general \emph{not} an incremental Lyapunov function.} \hfill \oprocend
\end{remark}

\begin{pfof}{Lemma \ref{Lem:HillMoylanIncremental}}
\emph{Sufficiency:} Let $\bar{x} \in \mathcal{E}_{\Sigma}$ be arbitrary, {\tb with associated equilibrium inputs/outputs given by \eqref{Eq:InputOutputMappings}}. Consider  the storage function candidate \eqref{Eq:Bregman}. It follows from Lemma \ref{Lem:Bregman} that $V_{\bar{x}}(\bar{x}) = 0$ and $V_{\bar{x}}(x) \geq 0$ for all $x \neq \bar{x}$. We compute that {\tb along system trajectories}
\begin{equation}\label{Eq:Vbar}
\begin{aligned}
\dot{V}_{\bar{x}} &= [\nabla V(x)-\nabla V(\bar{x})]^{\sf T}[f(x)+Gu]\\
%
%
%
%
&= [\nabla V(x)-\nabla V(\bar{x})]^{\sf T}[f(x)-f(\bar{x})]\\ 
&\quad + [\nabla V(x)-\nabla V(\bar{x})]^{\sf T}G(u-\bar{u})
\end{aligned}
\end{equation}
where we have used that $f(\bar{x}) + G\bar{u} = \vzeros[n]$ {\tb and, for notational simplicity, suppressed the time-dependence}. Adding the nonnegative quantity {\tb $\|\ell(x,\bar{x})+W(u-\bar{u})\|_2^2$} to the right-hand side of the dissipation rate, we obtain
$$
\begin{aligned}
\dot{V}_{\bar{x}} &\leq [\nabla V(x)-\nabla V(\bar{x})]^{\sf T}[f(x)-f(\bar{x})]\\
&\quad + {\tb \|\ell(x,\bar{x})\|_2^2} + [\nabla V(x)-\nabla V(\bar{x})]^{\sf T}G(u-\bar{u})\\
&\quad + 2{\tb \ell(x,\bar{x})}^{\sf T}W(u-\bar{u}) +  (u-\bar{u})^{\sf T}W^{\sf T}W(u-\bar{u})\,.
%
%
%
%
\end{aligned}
$$
Inserting \eqref{Eq:HillMoylanEID1} and \eqref{Eq:HillMoylanEID3}, we obtain
$$
\begin{aligned}
\dot{V}_{\bar{x}} &\leq  [h(x)-h(\bar{x})]^{\sf T}Q[h(x)-h(\bar{x})]\\
&\quad + [\nabla V(x)-\nabla V(\bar{x})]^{\sf T}G(u-\bar{u})\\
&\quad + 2{\tb \ell(x,\bar{x})}^{\sf T}W(u-\bar{u}) + (u-\bar{u})^{\sf T}\widehat{R}(u-\bar{u})\,,
%
%
%
%
\end{aligned}
$$
where $\widehat{R} = R + J^{\sf T}S + S^{\sf T}J + J^{\sf T}QJ$. {\tb Inserting \eqref{Eq:HillMoylanEID2} into the dissipation inequality, we find }
$$
\begin{aligned}
\dot{V}_{\bar{x}} &\leq [h(x)-h(\bar{x})]^{\sf T}Q[h(x)-h(\bar{x})]\\ 
&\quad + (u-\bar{u})^{\sf T}J^{\sf T}QJ(u-\bar{u}) \\
&\quad {\tb +} 2[h(x)-h(\bar{x})]^{\sf T}(QJ+S)(u-\bar{u})\\ 
&\quad + 2(u-\bar{u})^{\sf T}S^{\sf T}J(u-\bar{u}) +  (u-\bar{u})^{\sf T}R(u-\bar{u})\,.
\end{aligned}
$$
Inserting $h(x) = y - Ju$ and $h(\bar{x}) = \bar{y} - J\bar{u}$, collecting terms, and simplifying, one arrives at $\dot{V}_{\bar{x}} \leq {\sf w}(u-\bar{u},y-\bar{y})$ which shows the system is EID.

\smallskip

\emph{Necessity:} Assume $\Sigma$ is EID with supply rate ${\sf w}(u,y)$ and storage function \eqref{Eq:Bregman}, i.e., for each $\bar{x} \in \mathcal{E}_{\Sigma}$ it holds that $\dot{V}_{\bar{x}} \leq {\sf w}(u-\bar{u},y-\bar{y})$. Defining ${\sf d}_{\bar{x}}(x,u) := -\dot{V}_{\bar{x}} + {\sf w}(u-\bar{u},y-\bar{y})$, we find that
$$
\begin{aligned}
0 &\leq {\sf d}_{\bar{x}}(x,u) = -[\nabla V(x)-\nabla V(\bar{x})]^{\sf T}[f(x)+Gu]\\
&\qquad + (y-\bar{y})^{\sf T}Q(y-\bar{y}) + (u-\bar{u})^{\sf T}R(u-\bar{u})\\
&\qquad + 2(y-\bar{y})^{\sf T}S(u-\bar{u})
\end{aligned}
$$
Substituting for $y$ and $\bar{y}$, after some manipulation one obtains
{\tb
\begin{equation*}
\begin{aligned}
{\sf d}_{\bar{x}}(x,u) 
&= \begin{bmatrix}
1 \\ u-\bar{u}
\end{bmatrix}^{\sf T}\underbrace{\begin{bmatrix}
a(x,\bar{x}) & b(x)^{\sf T}-b(\bar{x})^{\sf T}\\
b(x)-b(\bar{x}) & \widehat{R}
\end{bmatrix}}_{:= \mathcal{D}(x,\bar{x})}\begin{bmatrix}
1 \\ u-\bar{u}
\end{bmatrix}
\end{aligned}
\end{equation*}
}
where
{\tb
\begin{equation}\label{Eq:abdef}
\begin{aligned}
a(x,\bar{x}) &= -[\nabla V(x)-\nabla V(\bar{x})]^{\sf T}[f(x)-f(\bar{x})]\,,\\
& \qquad + [h(x)-h(\bar{x})]^{\sf T}Q[h(x)-h(\bar{x})]\\
b(x)^{\sf T} &= -\frac{1}{2}\nabla V(x)^{\sf T}G + h(x)^{\sf T}(QJ+S)\,,
\end{aligned}
\end{equation}
and $\widehat{R}$ is as before.
Since ${\sf d}_{\bar{x}}(x,u) \geq 0$ for all $u$, we in fact have that $\mathcal{D}(x,\bar{x}) \succeq \vzeros[]$ for all $(x,\bar{x})$ \cite[Lemma 4.1.3]{AJvdS:16}; in particular then $a(x,\bar{x}) \geq 0$ and $\widehat{R} \succeq \vzeros[]$. For each pair $(x,\bar{x})$, the matrix $\mathcal{D}(x,\bar{x})$ may be factorized as
\begin{equation}\label{Eq:MathcalDFactor}
\mathcal{D}(x,\bar{x}) = \begin{bmatrix}
\tilde{\ell}(x,\bar{x})^{\sf T} \\ W(x,\bar{x})^{\sf T}
\end{bmatrix}\begin{bmatrix}
\tilde{\ell}(x,\bar{x}) & W(x,\bar{x})
\end{bmatrix}
\end{equation}
where $\map{\tilde{\ell}}{\mathcal{X} \times \mathcal{X}}{\real^k}$ and $\map{W}{\mathcal{X} \times \mathcal{X}}{\real^{k \times m}}$ for some nonnegative integer $k$.\footnote{{\tb For example, if one uses an SVD decomposition then $k$ can be chosen as the maximum rank of $\mathcal{D}(x,\bar{x})$ over $(x,\bar{x})$, and hence $k \leq m+1$. This is just one option though; see \cite[Chapter 4.1]{AJvdS:16} for some further discussion.}} 
It follows by equating blocks of $\mathcal{D}(x,\bar{x})$ that
\begin{subequations}\label{Eq:FactorEquated}
\begin{align}
\label{Eq:FactorEquated1}
\tilde{\ell}(x,\bar{x})^{\sf T}\tilde{\ell}(x,\bar{x}) &= a(x,\bar{x})\\
\label{Eq:FactorEquated2}
W(x,\bar{x})^{\sf T}\tilde{\ell}(x,\bar{x}) &= b(x)-b(\bar{x})\\
\label{Eq:FactorEquated3}
W(x,\bar{x})^{\sf T}W(x,\bar{x}) &= \widehat{R}
\end{align}
\end{subequations}
for all pairs $(x,\bar{x})$. We now show that without loss of generality, one may select $W(x,\bar{x}) = W$ as constant. From Lemma \ref{Lem:Orthogonal}, \eqref{Eq:FactorEquated3} holds if and only if $W(x,\bar{x}) = \mathcal{O}(x,\bar{x})W$ for an orthogonal matrix $\mathcal{O}(x,\bar{x}) \in \real^{k \times k}$ and a \emph{constant} matrix $W \in \real^{k \times m}$. Defining $\ell(x,\bar{x}) := \mathcal{O}(x,\bar{x})^{\sf T}\tilde{\ell}(x,\bar{x})$ and inserting these expressions into \eqref{Eq:FactorEquated}, the orthogonal matrices vanish and we find that
\begin{subequations}\label{Eq:FactorEquatedIndepX}
\begin{align}
\label{Eq:FactorEquatedIndepX1}
\ell(x,\bar{x})^{\sf T}\ell(x,\bar{x}) &= a(x,\bar{x})\\
\label{Eq:FactorEquatedIndepX2}
W^{\sf T}\ell(x,\bar{x}) &= b(x)-b(\bar{x})\\
\label{Eq:FactorEquatedIndepX3}
W^{\sf T}W &= \widehat{R}
\end{align}
\end{subequations}
which shows that we may indeed select $W(x,\bar{x}) = W$ independent of $(x,\bar{x})$. Substitution of the expressions for $a(x,\bar{x})$, $b(x)$ and $\widehat{R}$ into \eqref{Eq:FactorEquatedIndepX1}--\eqref{Eq:FactorEquatedIndepX3} immediately leads to the three equations \eqref{Eq:HillMoylanEID1}--\eqref{Eq:HillMoylanEID3}. To show the final statement, note from \eqref{Eq:abdef} that $a(x,x) = 0$, and hence it follows from \eqref{Eq:FactorEquatedIndepX2} that $\ell(x,x) = \vzeros[k]$. Using Lemma \ref{Lem:DifferenceFunction}, the equation \eqref{Eq:FactorEquatedIndepX2} holds if and only if
\begin{equation}\label{Eq:Cocycle}
W^{\sf T}(\ell(x_1,x_2) + \ell(x_2,x_3) + \ell(x_3,x_1)) = \vzeros[m]
\end{equation}
for any triple $(x_1,x_2,x_3)$. With $z \in \real^k$ as an auxiliary variable for brevity, observe that a particular solution of the equation $W^{\sf T}z = \vzeros[m]$ in \eqref{Eq:Cocycle} is $z_{\rm par} = \ell(x_1,x_2) + \ell(x_2,x_3) + \ell(x_3,x_1) = \vzeros[k]$. Using Lemma \ref{Lem:DifferenceFunction} once more, this implies that $\ell(x,\bar{x}) = l(x)-l(\bar{x})$ for an appropriate function $\map{l}{\mathcal{X}}{\real^k}$. Let $r := \dim(\ker(W^{\sf T}))$, let $t_1,\ldots,t_r \in \real^{k}$ be a basis for $\ker(W^{\sf T})$, and set $T := \begin{bmatrix}t_1 & \cdots & t_r\end{bmatrix}$. Then $W^{\sf T}T = \vzeros[]$ and the homogeneous solution to $W^{\sf T}z = \vzeros[m]$ can be written as
$$
z_{\rm hom} = T\left[q(x_1,x_2) + q(x_2,x_3) + q(x_3,x_1)\right]
$$
for some function $\map{q}{\mathcal{X}\times\mathcal{X}}{\real^r}$ satisfying $q(x,x) = \vzeros[r]$. Combining the particular and homogeneous solutions, it follows that we may take $\ell(x,\bar{x}) = l(x) - l(\bar{x}) + Tq(x,\bar{x})$, which completes the proof.
}
%
%
%
%
%
%
\end{pfof}

\medskip

{\tb The equation \eqref{Eq:HillMoylanEID3} is identical to the third Hill-Moylan condition \eqref{Eq:HillMoylan3}. When $W^{\sf T}$ has full column rank, then the final statement of Lemma \ref{Lem:HillMoylanIncremental} implies that $\ell(x,\bar{x})$ may always be chosen in the form $\ell(x,\bar{x}) = l(x)-l(\bar{x})$. In this case, the second equation \eqref{Eq:HillMoylanEID2} of Lemma \ref{Lem:HillMoylanIncremental} may be alternatively written as
$$
\begin{aligned}
\frac{1}{2}\nabla V(x)G &= h(x)^{\sf T}(QJ+S) - l(x)^{\sf T}W + \xi^{\sf T}
\end{aligned}
$$
for a constant vector $\xi \in \real^m$, which is quite similar to the second Hill-Moylan condition \eqref{Eq:HillMoylan2}.} As a special case of Lemma \ref{Lem:HillMoylanIncremental}, consider the supply rate ${\sf w}(u,y) = -y^{\sf T}y + \gamma^2 u^{\sf T}u$; this corresponds to $\Sigma$ having a finite $\mathscr{L}_2$-gain less than or equal to $\gamma$. The conditions of Lemma \ref{Lem:HillMoylanIncremental} for EID reduce to
\begin{equation}\label{Eq:NonlinearBoundedReal}
\begin{aligned}
[\nabla V(x)-\nabla V(\bar{x})]^{\sf T}&[f(x)-f(\bar{x})] +\|h(x)-h(\bar{x})\|_2^2\\ 
&=  - {\tb \|\ell(x,\bar{x})\|_2^2}\\
\frac{1}{2}[\nabla V(x)-\nabla V(\bar{x})]^{\sf T}G &= -[h(x)-h(\bar{x})]^{\sf T}J - {\tb \ell(x,\bar{x})}^{\sf T}W\\
\gamma^2 I_m - J^{\sf T}J  &= W^{\sf T}W
\end{aligned}
\end{equation}
which is an incremental nonlinear-bounded-real type result (c.f. \cite[Example 1]{HM:76}). The assumption for Lemma \ref{Lem:HillMoylanIncremental} that the input and throughput matrices $G$ and $J$ are state-independent allows for \eqref{Eq:HillMoylanEID2}--\eqref{Eq:HillMoylanEID3} {\tb to remain similar to the corresponding equations in the standard Hill-Moylan result of Lemma \ref{Lem:HillMoylan}. We do not pursue the extension to state-dependent input and throughout matrices here here, but see \cite[Section 3.1]{MA-CM-AP:16} for the case of scalar systems.}

\smallskip

{\tb In \cite[Example 3.1]{MA-CM-AP:16}, the single-input single-output scalar system $\dot{x} = f(x) + u$ with output $y = h(x)$ was shown to be equilibrium-independent passive if $f$ is continuous and decreasing and $h$ is continuous and increasing, with EIP storage function family parameterized as
\begin{equation}\label{Eq:Meissen}
V_{\bar{x}}(x) = \int_{\bar{x}}^x [h(z)-h(\bar{x})]\,\mathrm{d}z\,.
\end{equation}
Since $h$ is increasing, there exists a continuously differentiable convex function $\map{V}{\real}{\real}$ such that $h(x) = \nabla V(x)$, and \eqref{Eq:Meissen} can be seen as a special case of the construction \eqref{Eq:Bregman} used in Lemma \ref{Lem:HillMoylanIncremental}. Concerning the requirement that $f$ be decreasing, the following corollary of Lemma \ref{Lem:HillMoylanIncremental} generalizes this idea to higher-dimensional systems.

\smallskip

{\tb
\begin{corollary}{\bf (Equilibrium-Independent Passive Systems):}
Consider the square control-affine nonlinear system
\begin{equation}\label{Eq:PassiveSystem}
\begin{aligned}
\dot{x} &= f(x) + Gu\,, \quad y = G^{\sf T}\nabla V(x)
\end{aligned}
\end{equation}
where $\map{V}{\real^n}{\real}$ is continuously differentiable and strongly convex. If the mapping $-f \circ \nabla V^{-1}$ is monotone, then \eqref{Eq:PassiveSystem} is equilibrium-independent passive with storage function \eqref{Eq:Bregman}.
\end{corollary}

\begin{proof}.
Since $V$ is continuously differentiable and strongly convex, $x \mapsto \nabla V(x)$ is both maximally and strongly monotone, and is therefore a bijection on $\mathcal{X}$ \cite[Example 22.9]{HB-PC:11}. Therefore, $\tilde{f} := -f \circ \nabla V^{-1}$ is well-defined, and by assumption satisfies
\begin{equation}\label{Eq:fincr}
(x_1-x_2)^{\sf T}(\tilde{f}(x_1)-\tilde{f}(x_2)) \geq 0\,, \qquad x_1, x_2 \in \mathcal{X}\,.
\end{equation}
For the system \eqref{Eq:PassiveSystem} with supply rate $(Q,S,R) = (\vzeros[],\frac{1}{2}I_m,\vzeros[])$ one may quickly verify that \eqref{Eq:HillMoylanEID2}--\eqref{Eq:HillMoylanEID3} automatically hold with $W = \vzeros[]$, and therefore \eqref{Eq:HillMoylanEID1} holds if and only if
\begin{equation}\label{Eq:fincr2}
[\nabla V(x)-\nabla V(\bar{x})]^{\sf T}[f(x)-f(\bar{x})] \leq 0
\end{equation}
for all $(x,\bar{x}) \in \mathcal{X} \times \mathcal{E}_{\Sigma}$. Setting $x_1 := \nabla V(x)$, $x_2 = \nabla V(\bar{x})$, we see that \eqref{Eq:fincr} implies \eqref{Eq:fincr2}, which shows the result.
\end{proof}
}

}

\smallskip

{\tb
\begin{remark}{\bf (Computational Verification of EID):}
While appropriate functions $V(x)$ for the Lyapunov construction \eqref{Eq:Bregman} can sometimes be chosen for a system based on intuition, a suitable choice may not always be obvious. To verify the EID property for a given nonlinear system \eqref{Eq:NonlinearSystem} using Lemma \ref{Lem:HillMoylanIncremental}, in general one would seek to find a differentiable function $V(x)$ such is convex and establishes the dissipation inequality, i.e., 
$$
\begin{aligned}
[\nabla V(x)-\nabla V(\bar{x})]^{\sf T}(x-\bar{x}) &\geq 0\\
[\nabla V(x)-\nabla V(\bar{x})]^{\sf T}[f(x)+Gu] &\leq {\sf w}(u-\bar{u},y-\bar{y})
\end{aligned}
$$
for all $(x,\bar{x},u)$ with corresponding values for $(y,\bar{y},\bar{u})$. For LTI systems with quadratic storage functions, these constraints reduce to finding a symmetric matrix $P \succeq \vzeros[]$ such that the left-hand side of \eqref{Eq:KYP} is negative semidefinite; this is a linear matrix inequality. When $f(x)$ and $h(x)$ are polynomial functions, the search for a polynomial function $V(x)$ certifying EID can be cast as a sum-of-squares feasibility problem and solved via semidefinite programming; see \cite{MA-CM-AP:16} for further discussion. \hfill \oprocend
\end{remark}
}

%
%
%




\medskip

\subsection{Illustrative Examples}

{\tb Our first example illustrates the usefulness of the Lyapunov construction \eqref{Eq:Bregman}.}

\begin{example}{\bf (Port-Hamiltonian Systems):}\label{Example:Port}
A port-Hamiltonian system with {\tb state-independent input, dissipation, and interconnection matrices $G$, $\mathcal{R}$, and $\mathcal{J}$} takes the form
{\tb 
\begin{equation}\label{Eq:PortHamiltonian}
\begin{aligned}
\dot{x} &= [\mathcal{J}-\mathcal{R}]\nabla H(x) + Gu + d\\
y &= G^{\sf T}\nabla H(x)\,,
\end{aligned}
\end{equation}
where $\map{H}{\mathcal{X}}{\real_{\geq 0}}$ is a convex function, $d \in \real^n$ is an unknown constant disturbance, and $\mathcal{J},\mathcal{R} \in \real^{n \times n}$ satisfy $\mathcal{J} = -\mathcal{J}^{\sf T}$ and $\mathcal{R} = \mathcal{R}^{\sf T} \succeq \vzeros[]$. The system from Example \ref{Ex:TwoStateSystem} is a special case of the above port-Hamiltonian model. Forced equilibria are determined by
$$
\mathcal{E}_{\Sigma} = \setdef{\bar{x} \in\mathcal{X}}{\exists \bar{u} \in \mathcal{U}\,\,\mathrm{s.t.}\,\,[\mathcal{J}-\mathcal{R}]\nabla H(\bar{x}) + G\bar{u} + d = \vzeros[n]}\,,
$$
and therefore depend on the unknown disturbance $d$. With $V(x) = H(x)$, we have the EID storage candidate $V_{\bar{x}}(x) = H(x) - H(\bar{x}) - \nabla H(\bar{x})^{\sf T}(x-\bar{x})$, and a direct computation shows that
$$
\begin{aligned}
\dot{V}_{\bar{x}} &= -(\nabla H(x)-\nabla H(\bar{x}))^{\sf T}\mathcal{R}(\nabla H(x)-\nabla H(\bar{x}))\\ &\quad + (y-\bar{y})^{\sf T}(u-\bar{u})\,.
\end{aligned}
$$
Note that due to the incremental nature of the dissipation inequality, the unknown constant disturbance $d$ does not appear. The conditions of Lemma \ref{Lem:HillMoylanIncremental} are satisfied with $(Q,S,R) = (\vzeros[],\frac{1}{2}I_m,\vzeros[])$, $k = n$, $W = \vzeros[n \times m]$, and $\ell(x,\bar{x}) = \mathcal{R}^{\frac{1}{2}}[\nabla H(x)-\nabla H(\bar{x})]$. Suppose that we now wish to regulate the output of the system \eqref{Eq:PortHamiltonian} to a desired set point $\bar{y} \in \mathcal{Y}$; we assume there exists an assignable equilibrium $\bar{x} \in \mathcal{E}_{\Sigma}$ such that $\bar{y} = k_y(\bar{x})$. Consider the PI controller with input $e(t)$ and output $y_{\rm c}(t)$:
\begin{equation}\label{Eq:PIControl}
\dot{\zeta} = e\,, \quad y_{\rm c} = K_{\rm P}e + K_{\rm I}\zeta\,,
\end{equation}
where $K_{\rm P}, K_{\rm I} \succ \vzeros[]$. A quick calculation shows that with EID storage function $W_{\bar{\zeta}}(\zeta) = \frac{1}{2}\|\zeta-\bar{\zeta}\|_{K_{\rm I}}^2$, the PI controller \eqref{Eq:PIControl} satisfies the EID inequality
$$
\dot{W}_{\bar{\zeta}} = (y_{\rm c} - \bar{y}_{\rm c})^{\sf T}(e-\bar{e}) - (e-\bar{e})^{\sf T}K_{\rm P}(e-\bar{e})\,.
$$
With the negative feedback interconnection $e = y-\bar{y}$, $u = - y_{\rm c}$, we may add the EID dissipation inequalities to obtain $\dot{V}_{\bar{x}} + \dot{W}_{\bar{\zeta}} \leq - \|y-\bar{y}\|_{K_{\rm P}}^2$. From this point, one can argue in a standard fashion that $y(t) \rightarrow \bar{y}$; under additional assumptions (see Section \ref{Sec:Stability}) global exponential stability of the  corresponding equilibrium $\bar{x} \in \mathcal{E}_{\Sigma}$ can be obtained. This example illustrates the utility of the EID property, namely that tasks such as constant disturbance rejection, certification of subsystem input-output properties, and assessment of closed-loop stability become independent of the operating point being considered. For some recent extensions of these ideas to the case where $\mathcal{J}$ and $\mathcal{R}$ are state-dependent, see \cite{NM-PM-RO-AvdS:17}.
}
 \oprocend
\end{example}

\smallskip

{\tb The next example shows how one may work backwards from a desired supply rate using the algebraic conditions \eqref{Eq:HillMoylanEID}.}

\smallskip

\begin{example}{\bf (Gradient System w/ Feedthrough):}\label{Example:Grad}
Consider the square ($m=p=n$) system
\begin{equation}\label{Eq:Gradient}
\Sigma:\,\begin{cases}
\begin{aligned}
\tau\dot{x} &= -\nabla \phi(x) + gu\\
y &= gx + ju\,,
\end{aligned}
\end{cases}
\end{equation}
where $g, j > 0$, $\tau \succ \vzeros[]$ is diagonal, and $\map{\phi}{\real^n}{\real}$ is a $\mu$-strongly convex and differentiable function, i.e., there exists $\mu > 0$ such that
\begin{equation}\label{Eq:PhiConvex}
(x_1-x_2)^{\sf T}[\nabla \phi(x_1)-\nabla \phi(x_2)] \geq \mu \|x_1-x_2\|_2^2
\end{equation}
for all $x_1, x_2 \in \real^n$. We will use Lemma \ref{Lem:HillMoylanIncremental} to derive conditions under which \eqref{Eq:Gradient} is EID with respect to the supply rate ${\sf w}(u,y) = -\rho y^{\sf T}y - \nu u^{\sf T}u + u^{\sf T}y$, for values $\rho, \nu \geq 0$ to be determined. {\tb To begin,} the condition \eqref{Eq:HillMoylanEID3} becomes
{\tb
\begin{equation}\label{Eq:ExGradW}
W^{\sf T}W = (-\nu + j - \rho j^2)I_n.
\end{equation}
}
Assuming that
\begin{equation}\label{Eq:GradientFFCondition1}
j - \rho j^2 > \nu\,,
\end{equation}
the right-hand side {\tb of \eqref{Eq:ExGradW}} is positive definite, {\tb and} we may take $k = n$ and $W = (j - \rho j^2 - \nu)^{1/2}I_n$. With $V(x) = \frac{1}{2}x^{\sf T}\tau x$ {\tb and $G = \tau^{-1}g$}, the condition \eqref{Eq:HillMoylanEID2} reads (after substituting $W$) as
{\tb 
$$
g(x-\bar{x}) = \left(1 - 2\rho j\right)g(x-\bar{x}) - 2\sqrt{j - \rho j^2 - \nu}\,\ell(x,\bar{x})\,.
$$
}
We may therefore take $\ell(x,\bar{x}) = \beta (x-\bar{x})$, where
$$
\beta = \frac{-g\rho j}{\sqrt{j - \rho j^2 - \nu}}\,.
$$
{\tb
Finally, to establish that the system is EID, we can in fact enforce \eqref{Eq:HillMoylanEID1} as an inequality. With $f(x) = -\tau^{-1}\nabla \phi(x)$, after substitution of $\ell(x,\bar{x})$ \eqref{Eq:HillMoylanEID1} becomes
\begin{equation}\label{Eq:ExGradTemp}
\begin{aligned}
-(x-\bar{x})^{\sf T}&[\nabla \phi(x)-\nabla \phi(\bar{x})]\\ &\leq -\rho g^2\|x-\bar{x}\|_2^2  - \beta^2\|x-\bar{x}\|_2^2\,.
\end{aligned}
\end{equation}
}
{\tb Substituting \eqref{Eq:PhiConvex} into \eqref{Eq:ExGradTemp}, one finds that the required condition for \eqref{Eq:ExGradTemp} to be satisfied is $\mu \geq \rho g^2 + \beta^2$, or equivalently (after some algebra)}
\begin{equation}\label{Eq:GradientFFCondition2}
\rho \leq \frac{\mu}{g^2}\frac{\nu-j}{\nu - j - \mu j^2/g^2}\,.
\end{equation}
The two inequalities \eqref{Eq:GradientFFCondition1},\eqref{Eq:GradientFFCondition2} define the achievable set of EID dissipativity parameters $(\nu,\rho)$ {\tb as a function of $(\mu,j,g)$}; for $g = 1$, this set is plotted in Figure \ref{Fig:RhoNu}. Note that the input-passivity parameter $\nu$ can be increased by increasing the feedthrough $j$, but only at the expense of lowering the achievable size of the output-passivity parameter $\rho$. Moreover, the mere presence of non-zero feedthrough $j$ places limits on achievable values of $\rho$, irrespective of the convexity parameter $\mu$. 
\hfill \oprocend

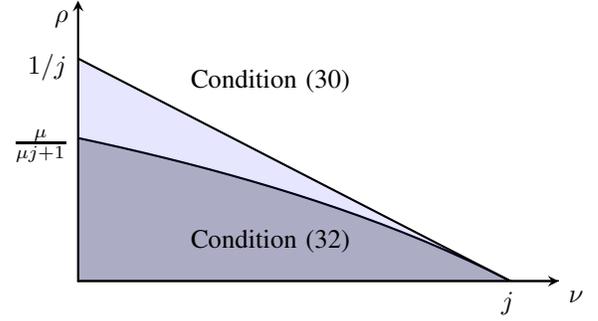
\begin{figure}[t] 
\centering 
\begin{tikzpicture}[every node/.style={scale=1}]

\begin{axis}[
enlargelimits=0,
ymin=0,
xmin=0,
xmax=1,
ymax=1.4,
ticks=none,
xticklabels=\empty,
yticklabels=\empty,
ylabel={$\rho$},
xlabel={$\nu$},
axis lines = left,
width=0.9\columnwidth,
height = 0.6\columnwidth,
]
    \addplot[domain=0:1.2*\feedforward,fill=gray!50] {\strongconvex*(x-\feedforward)/(x-\feedforward-\strongconvex*(\feedforward)^2)}\closedcycle;
        \addplot [thick,color=black,mark={},fill=blue, 
                    fill opacity=0.1]coordinates {
            (0, 1/\feedforward) 
            (\feedforward, 0)
			}\closedcycle;
            \node [rotate=0] at (axis cs:  .4,  1) {Condition \eqref{Eq:GradientFFCondition1}};
            \node [rotate=0] at (axis cs:  0.4,  .2) {Condition \eqref{Eq:GradientFFCondition2}};
    \end{axis}
    
\node [black, right] at (5.5,-0.3) {$j$};
\node [black, right] at (-0.8,2.85) {$1/j$};
\node [black, right] at (-1,1.8) {$\frac{\mu}{\mu j+1}$};
    
\end{tikzpicture}  
\caption{Feasible set for passivity parameters $(\nu,\rho)$ with $g = 1$.} 
\label{Fig:RhoNu}
\end{figure}

\end{example}

\smallskip

{\tb The final example of this section shows how EID may be used to assess the input-output performance of a common continuous-time optimization algorithm.}

\smallskip

\begin{example}{\bf (AHU Saddle-Point Algorithm):}\label{Example:AHU}
Consider the constrained optimization problem
\begin{equation}\label{Eq:ConvexEqualityConstrained}
\begin{aligned}
& \underset{z \in \mathbb{R}^{n_1}}{\text{minimize}}
& & \sum_{i=1}^{n_1} \phi_i(z_i) + \frac{1}{2}(Az-b)^{\sf T}K(Az-b)\\
& \text{subject to}
& & Az = b\,,
\end{aligned}
\end{equation}
where $A \in \real^{n_2 \times n_1}$ has full row rank, $K = K^{\sf T} \succ \vzeros[]$, $b \in \real^{n_2}$, and each map $\map{\phi_i}{\real}{\real}$ is $\mu_i$-strongly convex and differentiable. For simplicity, we set $\phi(z) = \sum_{i=1}^{n_1} \phi_i(z_i)$. The Lagrangian function is given by {\tb $L(z,\lambda) = \phi(z) + \frac{1}{2}(Az-b)^{\sf T}K(Az-b) + \lambda^{\sf T}(Az-b)$}, where $\lambda \in \real^{n_2}$ is a vector of dual variables (multipliers). To calculate the optimizer, the \emph{saddle-point} or \emph{primal-dual} algorithm \cite{JW-NE:11,DF-FP:10,JWSP:16f} performs gradient descent on the primal variables and gradient ascent on the dual variables, which reduces to
\begin{equation}\label{Eq:SaddlePointContinuous}
\dot{z} = -\nabla \phi(z) - A^{\sf T}K(Az-b) - A^{\sf T}\lambda\,,\quad \dot{\lambda} = Az-b\,,
\end{equation}
with composite state vector $x = (z,\lambda)$. Consider now the associated input/output system
\begin{equation}\label{Eq:SaddleIO}
\begin{aligned}
\dot{z} &= -\nabla \phi(z) - A^{\sf T}K(Az-b) - A^{\sf T}\lambda + u\\
\dot{\lambda} &= Az - b\\
y &= z
\end{aligned}
\end{equation}
We claim that \eqref{Eq:SaddleIO} is EID with respect to the supply rate ${\sf w}(u,y) = - y^{\sf T}y + \gamma^2 u^{\sf T}u$ for some $\gamma$ to be determined. The third condition in \eqref{Eq:NonlinearBoundedReal} gives that $W = \gamma I_n$, while with $V(x) = V(z,\lambda) = \frac{\alpha}{2}z^{\sf T}z + \frac{1}{2}\lambda^{\sf T}\lambda$, where $\alpha > 0$, the second condition in \eqref{Eq:NonlinearBoundedReal} yields $\ell(x,\bar{x}) = -\frac{\alpha}{2\gamma}(z-\bar{z})$. A quick computation shows that
$$
\begin{aligned}
&[\nabla V(x)-\nabla V(\bar{x})]^{\sf T}[f(x)-f(\bar{x})]\\
&= -\alpha(z-\bar{z})^{\sf T}[\nabla \phi(z)-\nabla \phi(\bar{z})]\ - \alpha (z-\bar{z})^{\sf T}A^{\sf T}KA(z-\bar{z})\\
&\leq -\alpha\,\lambda_{\rm min}(M + A^{\sf T}KA)\|z-\bar{z}\|_2^2\,,
\end{aligned}
$$
{\tb where $M = \mathrm{diag}(\mu_1,\ldots,\mu_{n_1})$ is the diagonal }matrix of convexity coefficients. The first condition in \eqref{Eq:NonlinearBoundedReal} therefore will hold with inequality sign if
$$
\alpha\lambda_{\rm min}(M + A^{\sf T}KA) \geq 1+\frac{\alpha^2}{4\gamma^2}\,.
$$
The choice $\alpha = 2\gamma^2\lambda_{\rm min}(M + A^{\sf T}KA)$ makes this inequality the tightest it can be, in which case it becomes
\begin{equation}\label{Eq:L2gain}
\gamma \geq \gamma^\star := \frac{1}{\lambda_{\rm min}(M + A^{\sf T}KA)}\,.
\end{equation}
We conclude that the system is EID with the given supply rate for any choice of $\gamma \geq \gamma^\star$, and in particular, with $\gamma = \gamma^\star$. This shows the system has equilibrium-independent $\mathscr{L}_2$-gain less than or equal to $\gamma^\star$. As discussed in \cite{JWSP:16f}, $\gamma^\star$ can in some situations be minimized by a judicious choice of the free parameter matrix $K$. {\tb When $K = \vzeros[]$, the $\mathscr{L}_2$-gain is limited only by the convexity parameters $M$ of the cost functions.} \hfill \oprocend \end{example}

\section{Stability of EID Systems}
\label{Sec:Stability}

This section presents internal and feedback stability results for continuous-time EID systems. The results are natural extensions of classical stability results for dissipative systems, but have not been stated in the literature.

\subsection{Internal Stability of EID Systems}

Standard stability results for dissipative systems proceed along the following lines. {\tb Consider the system $\Sigma$ in \eqref{Eq:NonlinearSystem} under the assumptions of Section \ref{Sec:ClassicDissipativityCT}.} If the system is dissipative with respect to the supply rate \eqref{Eq:SupplyNormal} with storage function $V(x)$ satisfying $V(\vzeros[n]) = 0$ and $V(x) > 0$ for $x \neq \vzeros[n]$, then with zero input the origin $x = \vzeros[n]$ is 1) stable if $Q \preceq \vzeros[]$, 2) asymptotically stable if $Q \prec \vzeros[]$ and $\Sigma$ is zero-state observable\footnote{The system is said to be zero-state observable if no solution of $\dot{x} = f(x)$ can stay within the set $\setdef{x}{h(x) = \vzeros[p]}$ other than $x(t) = \vzeros[n]$.}, and 3) globally asymptotically stable if $Q \prec \vzeros[]$, $\Sigma$ is zero-state observable, and $V(x) \rightarrow \infty$ as $\|x\|_2 \rightarrow \infty$ (i.e., $V$ is radially unbounded).\footnote{We will restrict our attention to cases where the storage function has a strict local minimum at the equilibrium point.}

\smallskip

We begin with an observability definition.

\smallskip

\begin{definition}{\bf (Equilibrium-Independent Observability):}\label{Def:Obsv}
The system \eqref{Eq:NonlinearSystem} is equilibrium-independent observable if, for every $\bar{x} \in \mathcal{E}_{\Sigma}$ with associated equilibrium input/output vectors $\bar{u} = k_u(\bar{x})$ and $\bar{y}=k_y(\bar{x})$, no trajectory of $\dot{x} = f(x) + G\bar{u}$ can remain within the set $\setdef{x \in \mathcal{X}}{h(x)+J\bar{u} = \bar{y}}$ other than the equilibrium trajectory $x(t) = \bar{x}$. 
\end{definition}

\smallskip

Definition \ref{Def:Obsv} is the natural extension of zero-state observability to EID systems, requiring that every forced system be ``zero-state'' observable. Compared to the general discussion of forced equilibria in Section \ref{Sec:ControlAffine}, Definition \ref{Def:Obsv} rules out the possibility that two distinct equilibria $\bar{x},\tilde{x} \in \mathcal{E}_{\Sigma}$ yield the same input/output pairs through \eqref{Eq:InputOutputMappings}. 

\smallskip

{\tb
\begin{proposition}{\bf (Observability \& Equilibrium Uniqueness):}\label{Prop:UniqueEq}
If the system \eqref{Eq:NonlinearSystem} is equilibrium-independent observable, then for a given equilibrium I/O pair $(\bar{u},\bar{y}) \in \mathcal{K}_{\Sigma}$, there is exactly one $\bar{x} \in \mathcal{E}_{\Sigma}$ satisfying \eqref{Eq:InputOutputMappings}.
\end{proposition}
\begin{proof}.
Let $(\bar{u},\bar{y}) \in \mathcal{K}_{\Sigma}$ be an equilibrium I/O pair and suppose that $\bar{x}$ and $\bar{x}^\prime$ are distinct points both satisfying \eqref{Eq:InputOutputMappings}. Then $x(t) = \bar{x}$ and $x(t) = \bar{x}^\prime$ are both trajectories of $\dot{x} = f(x) + G\bar{u}$, and both remain within the set $\setdef{x \in \mathcal{X}}{h(x)+J\bar{u}=\bar{y}}$, which contradicts equilibrium-independent observability. 
\end{proof}
}

%
%
%
%
%

\smallskip

\begin{lemma}{\bf (Internal Stability of EID Systems):}\label{Lem:InternalStability}
Suppose that a system $\Sigma$ satisfies the conditions of Lemma \ref{Lem:HillMoylanIncremental} with $V(x)$ strictly convex and $Q \prec \vzeros[]$. If $\Sigma$ is equilibrium-independent observable, then for every $\bar{x} \in \mathcal{E}_{\Sigma}$, $x = \bar{x}$ is a locally asymptotically stable equilibrium of the associated forced system $\dot{x} = f(x) + G\bar{u}$. Moreover, if $V(x)$ in Lemma \ref{Lem:HillMoylanIncremental} is strongly convex, then $\bar{x}$ is globally asymptotically stable.
\end{lemma}

\smallskip

\begin{proof}.
Fix an arbitrary $\bar{x} \in \mathcal{E}_{\Sigma}$ with associated $\bar{u}$ and $\bar{y}$, and let $V_{\bar{x}}(x)$ be as in Lemma \ref{Lem:HillMoylanIncremental}. Since $V_{\bar{x}}(\bar{x}) = 0$ and $V_{\bar{x}}(x) > 0$ for $x \neq \bar{x}$ (Lemma \ref{Lem:Bregman}(i)), $V_{\bar{x}}(x)$ is a Lyapunov candidate and satisfies the dissipation inequality
$$
\dot{V}_{\bar{x}}(x(t)) \leq (y(t)-\bar{y})^{\sf T}Q(y(t)-\bar{y})
$$
along trajectories of the forced system $\dot{x} = f(x) + G\bar{u}$. Since $Q \prec \vzeros[]$, there exists an $\alpha > 0$ such that $\dot{V}_{\bar{x}}(x(t)) \leq -\alpha \|y(t)-\bar{y}\|_2^2$. Standard arguments then show that $y(t) = h(x(t)) + J\bar{u}$ converges to $\bar{y}$ which, due to equilibrium-independent observability, implies that $x(t) \rightarrow \bar{x}$ showing local asymptotic stability. When $V(x)$ is $\mu$-strongly convex, Lemma \ref{Lem:Bregman}(ii) shows that $V_{\bar{x}}(x) \geq \frac{\mu}{2}\|x-\bar{x}\|_2^2$, and thus $V_{\bar{x}}(x)$ is radially unbounded; standard results (e.g., \cite[Corollary 2.2]{RS-MJ-PK:97}) then yield global asymptotic stability of $\bar{x}$.
\end{proof}

\smallskip

Variations on this result are possible, for example, by weakening the observability requirement to an appropriate notion of equilibrium-independent detectability; we omit the details.


\subsection{Interconnection and Feedback Stability}
\label{Sec:FeedbackStability}

Consider now two control-affine systems
\begin{equation*}
\Sigma_i:\,\begin{cases}
\begin{aligned}
\dot{x}_i &= f_i(x_i) + G_iu_i\\
y_i &= h_i(x_i) + J_iu_i
\end{aligned}
\end{cases}
\end{equation*}
$i \in \{1,2\}$, with compatible input/output {\tb spaces $\mathcal{U}_1 = \mathcal{Y}_2 = \real^m$ and $\mathcal{Y}_1 = \mathcal{U}_2 = \real^p$}, subject to the negative feedback interconnection of Figure \ref{Fig:Feedback}:
$$
u_1 = v_1 - y_2\,, \quad u_2 = v_2 + y_1\,.
$$
As is standard, we assume that $I_p + J_1J_2$ is nonsingular, which ensures that the feedback interconnection is well-posed. Given constant input vectors $\bar{v}_1, \bar{v}_2$, the conditions for a forced equilibrium of the closed-loop system are that 
\begin{equation}\label{Eq:ClosedLoopEquilibriumInclusions}
\begin{aligned}
\bar{y}_1 &\in \mathcal{K}_{\Sigma_1}(\bar{v}_1-\bar{y}_2)\\
\bar{y}_2 &\in \mathcal{K}_{\Sigma_2}(\bar{v}_2+\bar{y}_1)\,.
\end{aligned}
\end{equation} 
Lemma \ref{Lem:Intersection} in the appendix presents sufficient conditions for square EID systems which guarantee that for any pair of constant inputs $(\bar{v}_1,\bar{v}_2)$, the simultaneous inclusions \eqref{Eq:ClosedLoopEquilibriumInclusions} are uniquely solvable for corresponding outputs $(\bar{y}_1,\bar{y}_2)$. We will make explicit use of this lemma in Section \ref{Sec:Absolute}. For now, we simply assume that \eqref{Eq:ClosedLoopEquilibriumInclusions} yields a well-defined equilibrium input-output relation {\tb $\mathcal{K}_{\Sigma_{\rm cl}} \subset (\mathcal{U}_1\times\mathcal{U}_2) \times (\mathcal{Y}_1\times\mathcal{Y}_2)$} for the closed-loop system.

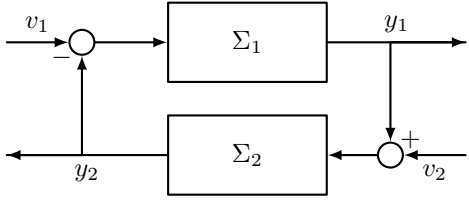
\begin{figure}[ht!]
\begin{center}
\tikzstyle{block} = [draw, fill=white, rectangle, 
    minimum height=3em, minimum width=6em]
    \tikzstyle{dzblock} = [draw, fill=white, rectangle, minimum height=3em, minimum width=4em,
path picture = {
\draw[thin, black] ([yshift=-0.1cm]path picture bounding box.north) -- ([yshift=0.1cm]path picture bounding box.south);
\draw[thin, black] ([xshift=-0.1cm]path picture bounding box.east) -- ([xshift=0.1cm]path picture bounding box.west);
\draw[very thick, black] ([xshift=-0.5cm]path picture bounding box.east) -- ([xshift=0.5cm]path picture bounding box.west);
\draw[very thick, black] ([xshift=-0.5cm]path picture bounding box.east) -- ([xshift=-0.1cm, yshift=+0.4cm]path picture bounding box.east);
\draw[very thick, black] ([xshift=+0.5cm]path picture bounding box.west) -- ([xshift=+0.1cm, yshift=-0.4cm]path picture bounding box.west);
}]
\tikzstyle{sum} = [draw, fill=white, circle, node distance=1cm]
\tikzstyle{input} = [coordinate]
\tikzstyle{output} = [coordinate]
\tikzstyle{pinstyle} = [pin edge={to-,thin,black}]

\begin{tikzpicture}[auto, scale = 0.6, node distance=2cm,>=latex', every node/.style={scale=1}]
    \node [input, name=input] {};
    \node [sum, right of=input] (sum) {};
    \node [block, right of=sum, node distance=2.2cm] (system) {$\Sigma_1$};
    \node [block, below of=system, node distance=1.5cm] (nonlin) {$\Sigma_2$};
    \node [output, right of=system, node distance=2.9cm] (output) {};
    \node [output, left of=nonlin, node distance=3.2cm] (output2) {};
    \node [input, name=input2, right of=nonlin, node distance=2.9cm] {};
    \node [sum, left of=input2] (sum2) {};

    \draw [thick, -latex] (input) -- node {$v_1$} (sum);
    \draw [thick, -latex] (input2) -- node {$v_2$} (sum2);
    \draw [thick, -latex] (sum2) -- (nonlin);
    \draw [thick, -latex] (sum) -- (system);
    \draw [thick, -latex] (system) -- node [name=y] {$y_1$}(output);
    \draw [thick, -latex] (output) -| node[pos=0.99] {$+$} 
        node [near end] {} (sum2);
    \draw [thick, -latex] (nonlin) -- node [name=y2] {$y_2$} (output2);
    \draw [thick, -latex] (output2) -| node[pos=0.99] {$-$} 
        node [near end] {} (sum);
\end{tikzpicture}
\end{center}
\caption{Feedback interconnection of two EID systems.}
\label{Fig:Feedback}
\end{figure}

\smallskip

\begin{theorem}{\bf (Dissipativity and Stability of EID Feedback Systems):}\label{Thm:EIDFeedback} Consider the feedback interconnection of Figure \ref{Fig:Feedback}. Suppose that
\begin{itemize}
\item $\Sigma_1$ and $\Sigma_2$ satisfy the conditions of Lemma \ref{Lem:HillMoylanIncremental} with convex functions $V_{i}(x_i)$ and supply parameters $(Q_i,S_i,R_i)$ for $i \in \{1,2\}$;
\item for every constant $(\bar{v}_1, \bar{v}_2)$, the inclusions \eqref{Eq:ClosedLoopEquilibriumInclusions} possess a unique solution $(\bar{y}_1,\bar{y}_2)$, with $(\bar{x}_1,\bar{x}_2) \in \mathcal{E}_{\Sigma_{\rm cl}} :=  \mathcal{E}_{\Sigma_1} \times \mathcal{E}_{\Sigma_2}$ being a corresponding closed-loop equilibrium point.
\end{itemize}
Then for any $\kappa > 0$, the closed-loop system with inputs $(v_1,v_2)$ and outputs $(y_1,y_2)$ is EID with supply parameters
$$
\begin{aligned}
Q_{\rm cl} &= \begin{bmatrix}
Q_1 + \kappa R_2 & -S_1 + \kappa S_2^{\sf T}\\
-S_1^{\sf T} + \kappa S_2 & R_1 + \kappa Q_2
\end{bmatrix}\\
S_{\rm cl} &= \begin{bmatrix}
S_1 & R_2\\
R_1 & S_2
\end{bmatrix}\,, \qquad R_{\rm cl} = \begin{bmatrix}
R_1 & \vzeros[]\\
\vzeros[] & R_2
\end{bmatrix}\,.
\end{aligned}
$$
and storage function $V_{\bar{x}}(x) = V_{1,\bar{x}_1}(x_1) + \kappa V_{2,\bar{x}_2}(x_2)$. Moreover, if 
\begin{itemize}
\item $V_1(x_1)$ and $V_2(x_2)$ are strictly convex,
\item $\Sigma_1$ and $\Sigma_2$ are equilibrium-independent observable, and
\item there exists $\kappa > 0$ such that $Q_{\rm cl} \prec \vzeros[]$, 
\end{itemize}
then $\bar{x} = (\bar{x}_1,\bar{x}_2) \in \mathcal{E}_{\Sigma_{\rm cl}}$ with associated constant inputs $(\bar{v}_1,\bar{v}_2) = k_u(\bar{x})$ is the unique closed-loop equilibrium point for the constant inputs $(v_1,v_2) = (\bar{v}_1,\bar{v}_2)$ and is locally asymptotically stable. If the respective functions $V_1(x_1)$ and $V_2(x_2)$ from Lemma \ref{Lem:HillMoylanIncremental} are strongly convex, then the previous statement is strengthened to global asymptotic stability of $\bar{x}$.
\end{theorem}

\smallskip

\begin{proof}.
A simple calculation shows the closed-loop system $\Sigma_{\rm cl}$ may be written as $\dot{x} = f(x) + Gv$, $y = h(x) + Jv$, where
$$
\begin{aligned}
f &= \begin{bmatrix}
f_1(x_1) - G_1(h_2(x_2)+J_2h_1(x_1))\\
f_2(x_2) + G_2(h_1(x_1)+J_1h_2(x_2))
\end{bmatrix}, \quad J = \begin{bmatrix}
J_1 & \vzeros[]\\
\vzeros[] & J_2
\end{bmatrix}\\
G &= \begin{bmatrix}
G_1 & -G_1J_2 \\ 
G_2J_1 & G_2
\end{bmatrix}, \quad h = 
\begin{bmatrix}
I_p  & -J_1\\ J_2 & I_m
\end{bmatrix}\begin{bmatrix}
h_1(x_1)\\h_2(x_2)
\end{bmatrix}
\end{aligned}
$$
If $G_1^{\perp}$ and $G_2^{\perp}$ are the full-rank left annihilators of $G_1$ and $G_2$, then $G^{\perp} = \mathrm{blkdiag}(G_1^{\perp},G_2^{\perp})$ serves as a full-rank left annihilator for $G$, and 
$$
\begin{aligned}
\mathcal{E}_{\Sigma_{\rm cl}} &= \setdef{x=(x_1,x_2)}{G^{\perp}f(x) = \vzeros[n_1+n_2]}\\
&= \setdef{x}{G_1^{\perp}f_1(x_1) = \vzeros[n_1]\,\,\text{and}\,\,\,G_2^{\perp}f_2(x_2) = \vzeros[n_2]}\\
&= \mathcal{E}_{\Sigma_1} \times \mathcal{E}_{\Sigma_2}\,.
\end{aligned}
$$
For any $\bar{x} \in \mathcal{E}_{\Sigma_{\rm cl}}$, taking $V_{\bar{x}}(x) = V_{1,\bar{x}_1}(x_1) + \kappa V_{2,\bar{x}_2}(x_2)$ and differentiating leads immediately to an EID dissipation inequality $\dot{V}_{\bar{x}}(x) \leq {\sf w}(v-\bar{v},y-\bar{y})$ with parameters $(Q_{\rm cl},S_{\rm cl},R_{\rm cl})$ as given. The final statement on stability follows by applying Lemma \ref{Lem:InternalStability} to the closed-loop system. 
\end{proof}

\subsection{Equilibrium-Independent Absolute Stability}
\label{Sec:Absolute}

We now consider the feedback system shown in Figure \ref{Fig:Absolute}, consisting of a square ($\mathcal{U}=\mathcal{Y}=\real^m$) system $\Sigma$ in feedback with a static nonlinear element $\map{\psi}{\real^m}{\real^m}$; we assume $\psi$ is sufficiently smooth to ensure well-defined closed-loop trajectories.

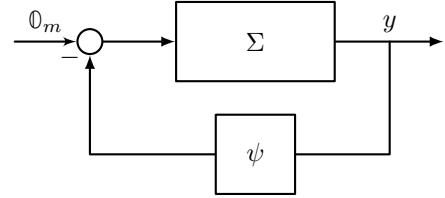
\begin{figure}[ht!]
\begin{center}
\tikzstyle{block} = [draw, fill=white, rectangle, 
    minimum height=3em, minimum width=6em]
    \tikzstyle{hold} = [draw, fill=white, rectangle, 
    minimum height=3em, minimum width=3em]
    \tikzstyle{dzblock} = [draw, fill=white, rectangle, minimum height=3em, minimum width=4em,
path picture = {
\draw[thin, black] ([yshift=-0.1cm]path picture bounding box.north) -- ([yshift=0.1cm]path picture bounding box.south);
\draw[thin, black] ([xshift=-0.1cm]path picture bounding box.east) -- ([xshift=0.1cm]path picture bounding box.west);
\draw[very thick, black] ([xshift=-0.5cm]path picture bounding box.east) -- ([xshift=0.5cm]path picture bounding box.west);
\draw[very thick, black] ([xshift=-0.5cm]path picture bounding box.east) -- ([xshift=-0.1cm, yshift=+0.4cm]path picture bounding box.east);
\draw[very thick, black] ([xshift=+0.5cm]path picture bounding box.west) -- ([xshift=+0.1cm, yshift=-0.4cm]path picture bounding box.west);
}]
\tikzstyle{sum} = [draw, fill=white, circle, node distance=1cm]
\tikzstyle{input} = [coordinate]
\tikzstyle{output} = [coordinate]
\tikzstyle{pinstyle} = [pin edge={to-,thin,black}]

\begin{tikzpicture}[auto, scale = 0.6, node distance=2cm,>=latex', every node/.style={scale=1}]
    \node [input, name=input] {};
    \node [sum, right of=input] (sum) {};
    \node [block, right of=sum, node distance=2.2cm] (system) {$\Sigma$};
    \node [hold, below of=system, node distance=1.5cm] (nonlin)  {$\psi$};
    \node [output, right of=system, node distance=2.5cm] (output) {};

    \draw [draw,->] (input) -- node {$\vzeros[m]$} (sum);
    \draw [thick, -latex] (sum) -- (system);
    \draw [thick, -latex] (system) -- node [name=y] {$y$}(output);
    \draw [thick, -latex] (y) |- (nonlin) -| node[pos=0.99] {$-$} 
        node [near end] {} (sum);
\end{tikzpicture}
\end{center}
\caption{System with static feedback nonlinearity.}
\label{Fig:Absolute}
\end{figure}

Classically, the absolute stability problem is to determine conditions under which the feedback system in Figure \ref{Fig:Absolute} is internally stable for all memoryless nonlinearities $\psi$ satisfying a sector condition. Crucially, in the standard formulation, $\Sigma$ is assumed to have an equilibrium point at the origin, and $\psi$ is assumed to satisfy $\psi(\vzeros[m]) = \vzeros[m]$; these assumptions ensure that the feedback interconnection has an unforced equilibrium point at the origin.\footnote{Typically $\Sigma$ is further assumed to be an LTI system.} The development of equilibrium-independent dissipativity allows us to consider a sensible variant on this problem, where rather than being \emph{assumed}, the existence of a closed-loop equilibrium point is \emph{inferred} from the EID properties of the subsystems. For simplicity of exposition, we assume that $J = \vzeros[]$ ($\Sigma$ has no feedthrough).

\medskip

\begin{theorem}{\bf (Equilibrium-Independent Circle Criterion):}\label{Thm:Absolute}
Consider the feedback system in Figure \ref{Fig:Absolute}, where $\Sigma$ is square ($m=p$) and is equilibrium-independent observable. Assume that
\begin{enumerate}
\item[(i)] the nonlinearity $\map{\psi}{\real^m}{\real^m}$ satisfies the incremental dissipation inequality \eqref{Eq:EIDNonlinearity}, with parameters\footnote{Equivalently, $\psi$ satisfies the incremental sector condition
$$
[\psi(z_2)-\psi(z_1) - K_1(z_2-z_1)]^{\sf T}[\psi(z_2)-\psi(z_1) - K_2(z_2-z_1)] \leq 0\,.
$$}
\begin{equation}\label{Eq:PsiDiss}
(Q_{\psi},S_{\psi},R_{\psi}) = \left(-I_m, \frac{K_1+K_2}{2}, -K_1K_2\right)\,,
\end{equation}
{\tb where $K_1, K_2$ are diagonal and $K = K_2 - K_1 \succ \vzeros[]$};
\item[(ii)] the system
\begin{equation}\label{Eq:NonlinearSystemLoop}
\Sigma^\prime:\,\begin{cases}
\begin{aligned}
\dot{x} &= f(x) - GK_1h(x) + Gu_{\ell}\\
y_{\ell} &= K h(x) + u_{\ell}
\end{aligned}
\end{cases}
\end{equation}
is EID, satisfying Lemma \ref{Lem:HillMoylanIncremental} with $V(x)$ strictly convex and supply rate \eqref{Eq:SupplyNormal}, with parameters
\begin{equation}\label{Eq:SigmaPrimeSupplyParameters}
(Q_{\Sigma^\prime},S_{\Sigma^\prime},R_{\Sigma^\prime}) = \left(-\varepsilon I_m,\frac{1}{2}I_m,\vzeros[]\right)
\end{equation}
for some $\varepsilon > 0$\,.
\end{enumerate}
Then the closed-loop system {\tb possesses a unique and locally asymptotically stable equilibrium point. If $V(x)$ is strongly convex, then the equilibrium is globally asymptotically stable.}
\end{theorem}

\begin{proof}.
Through a standard loop transformation (see, e.g., \cite[Pg. 267]{HKK:02}), we may transform the feedback interconnection of Figure \ref{Fig:Absolute} to the feedback interconnection in Figure \ref{Fig:AbsoluteTransformed}. The new nonlinearity $\map{\psi^\prime}{\real^m}{\real^m}$ in the feedback path satisfies the incremental dissipation inequality \eqref{Eq:EIDNonlinearity} with parameters $(Q_{\psi^\prime},S_{\psi^\prime},R_{\psi^\prime}) = \left(\vzeros[], \frac{1}{2}I_m, \vzeros[]\right)$\,, 
i.e., $\psi^\prime$ is monotone \cite[Pg. 233]{HKK:02}. 

%

\begin{figure}[ht!]
\begin{center}
\tikzstyle{block} = [draw, fill=white, rectangle, 
    minimum height=3em, minimum width=6em]
    \tikzstyle{hold} = [draw, fill=white, rectangle, 
    minimum height=3em, minimum width=3em]
    \tikzstyle{hold2} = [draw, fill=white, rectangle, 
    minimum height=2em, minimum width=3em]
    \tikzstyle{dzblock} = [draw, fill=white, rectangle, minimum height=3em, minimum width=4em,
path picture = {
\draw[thin, black] ([yshift=-0.1cm]path picture bounding box.north) -- ([yshift=0.1cm]path picture bounding box.south);
\draw[thin, black] ([xshift=-0.1cm]path picture bounding box.east) -- ([xshift=0.1cm]path picture bounding box.west);
\draw[very thick, black] ([xshift=-0.5cm]path picture bounding box.east) -- ([xshift=0.5cm]path picture bounding box.west);
\draw[very thick, black] ([xshift=-0.5cm]path picture bounding box.east) -- ([xshift=-0.1cm, yshift=+0.4cm]path picture bounding box.east);
\draw[very thick, black] ([xshift=+0.5cm]path picture bounding box.west) -- ([xshift=+0.1cm, yshift=-0.4cm]path picture bounding box.west);
}]
\tikzstyle{sum} = [draw, fill=white, circle, node distance=1cm]
\tikzstyle{input} = [coordinate]
\tikzstyle{output} = [coordinate]
\tikzstyle{pinstyle} = [pin edge={to-,thin,black}]

\begin{tikzpicture}[auto, scale = 0.6, node distance=2cm,>=latex', every node/.style={scale=1}]
    \node [input, name=input] {};
    \node [sum, right of=input] (sum) {};
    \node [block, right of=sum, node distance=3.2cm] (system) {$\Sigma^\prime$};
    \node [hold2, below of=system, node distance=1.5cm, xshift=-0.8cm] (nonlin)  {$\psi$};
    \node [hold2, right of=nonlin, node distance=2cm] (kinv)  {$K^{-1}$};
    \node [sum, right of=kinv, node distance=1.3cm] (sum3)  {$$};
    \node [hold2, below of=nonlin, node distance=1.3cm] (k1)  {$K_1$};
    \node [sum, left of=nonlin, node distance=1.3cm] (sum2) {};
    \node [output, right of=system, node distance=3.8cm] (output) {};

    \draw [draw,->] (input) -- node {$\vzeros[m]$} (sum);
    \draw [thick, -latex] (sum) -- node {$u_{\ell}$} (system);
    \draw [thick, -latex] (system) -- node [name=y] {$y_{\ell}$}(output);
    \draw [thick, -latex] (sum2) -| node[pos=0.98] {$-$} (sum);
    \draw [thick, -latex] (k1) -| node[pos=0.98] {$-$} (sum2);
    \draw [thick, -latex] (nonlin) -- (sum2);
    \draw [thick, -latex] ([xshift=0.8cm]nonlin.east) |- (k1);
    \draw [thick, -latex] (kinv) -- (nonlin);
    \draw [thick, -latex] (sum3) -- (kinv);
    \coordinate [below of=k1, node distance=0.6cm] (tmp);
    \draw [thick, -latex] ([xshift=-0.5cm]sum2.west) |- (tmp) -| node[pos=0.98] {$+$} (sum3);
    \draw [thick, -latex] ([xshift=-0.7cm]output.west) |- (sum3);
     \draw[black, dotted] ([yshift=10mm,xshift=5mm]sum3.east)|-([yshift=-5mm]tmp.east)-|([yshift=0mm,xshift=-10mm]sum2.west)|-([yshift=10mm,xshift=5mm]sum3.east);
    \node[black] at (12.5,-5)  {$\psi^\prime$};
\end{tikzpicture}
\end{center}
\caption{Loop-transformed feedback system.}
\label{Fig:AbsoluteTransformed}
\end{figure}
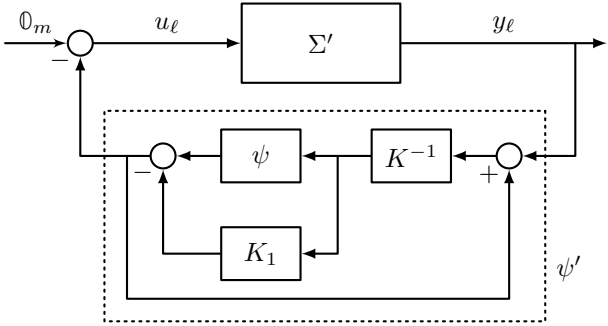

We first address the equivalence of equilibria between the two feedback loops, and the existence of an equilibrium point. Equilibria $\bar{x}$ of Figure \ref{Fig:Absolute} are determined by
\begin{equation}\label{Eq:MyFeedbackEquilibria}
\begin{aligned}
\vzeros[n] &= f(\bar{x}) - G\psi(\bar{y})\\
\bar{y} &= h(\bar{x})
\end{aligned} \,\, \Leftrightarrow \,\, \vzeros[n] = f(\bar{x}) - G\psi(h(\bar{x}))\\
\end{equation}
while equilibria $\tilde{x}$ of Figure \ref{Fig:AbsoluteTransformed} are determined by
\begin{subequations}\label{Eq:TransformedEquilibria}
\begin{align}\label{Eq:TransformedEquilibria1}
\vzeros[n] &= f(\tilde{x}) - GK_1h(\tilde{x}) -G\psi^\prime(\tilde{y}_{\ell})\\
\label{Eq:TransformedEquilibria2}
\tilde{y}_{\ell} &= Kh(\tilde{x}) - \psi^\prime(\tilde{y}_{\ell})
\end{align}
\end{subequations}
{\tb Since $\psi^\prime$ is continuous and monotone, the mapping $z \mapsto z + \psi^{\prime}(z)$ is continuous and $1$-strongly monotone, and therefore for every $b \in \real^{m}$ the equation $b = \gamma(z) := z + \psi^\prime(z)$ has a unique solution; we denote this solution by $z = \gamma^{-1}(b)$. It follows that \eqref{Eq:TransformedEquilibria2} may be uniquely solved for $\tilde{y}_{\ell} = \gamma^{-1}(Kh(\tilde{x}))$, and \eqref{Eq:TransformedEquilibria} is therefore equivalent to the single equation}
\begin{equation}\label{Eq:LTEquilibriumReduced}
\vzeros[n] = f(\tilde{x}) - GX(h(\tilde{x}))\,,
\end{equation}
where $X(h) := K_1h +\psi^\prime(\gamma^{-1}(Kh))$. {\tb Comparing \eqref{Eq:MyFeedbackEquilibria} and \eqref{Eq:LTEquilibriumReduced}, equivalence of equilibria will follow if $\psi = X$.} To show this, note from Figure \ref{Fig:AbsoluteTransformed} that $\psi^\prime$ is defined by
$$
\psi^\prime(z^\prime) = \psi(K^{-1}(z^\prime + \psi^\prime(z^\prime))) - K_1K^{-1}(z^\prime + \psi^\prime(z^\prime))\,.
$$
Substituting $\gamma(z^\prime) = z^\prime + \psi^\prime(z^\prime)$, we find that
$$
\psi^\prime(z^\prime) = \psi(K^{-1}\gamma(z^\prime)) - K_1K^{-1}\gamma(z^\prime)\,.
$$
Changing variables now to $h := \gamma^{-1}(Kz^\prime)$, this further simplifies to
$$
\psi^\prime(\gamma^{-1}(Kh)) = \psi(h) - K_1h
$$
from which it follows by comparison that $\psi(h) = X(h)$. {\tb Therefore, the equilibrium sets of the two feedback systems are equal.} To address existence and uniqueness of an equilibrium point, note that since $\psi^\prime$ is a continuous monotone function, $\mathcal{K}_{\psi^\prime} = \psi^\prime$ is maximally monotone (Lemma \ref{Lem:MaxMonotone}). Moreover, since $\Sigma^\prime$ is EID with supply rate parameters \eqref{Eq:SigmaPrimeSupplyParameters}, $\mathcal{K}_{\Sigma^\prime}$ is $\varepsilon$-cocoercive, and is therefore maximally monotone (Lemma \ref{Lem:MaxMonotone}). Applying Lemma \ref{Lem:Intersection} with $\mathcal{K}_{\Sigma_1} = \mathcal{K}_{\Sigma^\prime}$ and $\mathcal{K}_{\Sigma_2} = \mathcal{K}_{\psi^\prime}$, we conclude that the closed-loop system in Figure \ref{Fig:AbsoluteTransformed} possesses a unique equilibrium I/O pair $(\bar{u}_{\ell},\bar{y}_{\ell}) \in \mathcal{K}_{\Sigma^\prime}$ with $-\bar{u}_{\ell} = \psi^\prime(\bar{y}_{\ell})$. Therefore, by definition, there exists an associated equilibrium point $\bar{x} \in \mathcal{E}_{\Sigma^\prime}$, {\tb and this equilibrium point is unique by Proposition \ref{Prop:UniqueEq}. By the previous arguments on equivalence of equilibria between the systems, $\bar{x} \in \mathcal{E}_{\Sigma}$ as well.}

Using the EID storage function $V_{\bar{x}}(x)$, we compute that
$$
\begin{aligned}
\dot{V}_{\bar{x}} &\leq -\varepsilon \|y_{\ell}-\bar{y}_{\ell}\|_2^2 + (y_{\ell}-\bar{y}_{\ell})^{\sf T}(u_{\ell}-\bar{u}_{\ell})\\
&= -\varepsilon \|y_{\ell}-\bar{y}_{\ell}\|_2^2 - (y_{\ell}-\bar{y}_{\ell})^{\sf T}(\psi^\prime(y_{\ell})-\psi^{\prime}(\bar{y}_{\ell}))\\
&\leq -\varepsilon \|y_{\ell}-\bar{y}_{\ell}\|_2^2\,,
\end{aligned}
$$
where we have used that $\psi^\prime$ is monotone; the rest of the result follows from Lemma \ref{Lem:InternalStability}.
\end{proof}

\medskip

\begin{example}{\bf (SMIB Power System):}\label{Ex:SMIB}
Consider the single-machine infinite-bus (SMIB) power system model
$$
\begin{aligned}
\dot{\theta} &= \omega\,,\\
M\dot{\omega} &= P_{\rm m} - bV^2 \sin(\theta) - D\omega + u\,,\\
y &= \omega
\end{aligned}
$$
where $\theta \in \real$ is the rotor angle, $\omega \in \real$ is the generator frequency, $P_{\rm m} \in \real$ is the mechanical power, and $M, D, b, V > 0$; let $y = \omega$ be the output. {\tb By inspection,} the set of assignable equilibrium points is
$$
\mathcal{E}_{\Sigma} = \setdef{(\bar{\theta},\bar{\omega})}{\bar{\omega} = 0,\,\bar{\theta} \in \real}\,,
$$
with corresponding input $\bar{u} = k_u(\bar{\theta}) = bV^2\sin(\bar{\theta}) - P_{\rm m}$. {\tb For a fixed $\Gamma \in [0,\pi/2)$, we restrict our attention to equilibria in the set $\Theta(\Gamma) \times \{0\} \subset \mathcal{E}_{\Sigma}$ where $\Theta(\Gamma) = \setdef{\theta}{|\theta| \leq \Gamma}$; this ensures that the nonlinearity $\sin(\cdot)$ is strongly monotone in a neighbourhood of any equilibrium $\bar{\theta} \in \Theta(\Gamma)$. We further assume that $|P_{\rm m}| < bV^2\sin(\Gamma)$, which is {\tb necessary and sufficient} for the existence of an equilibrium $\bar{\theta} \in \Theta(\Gamma)$ when $u = 0$.}
Consider now the frequency feedback control $u = -\psi(\omega)$, where $\map{\psi}{\real}{\real}$ is incrementally in the sector $[\alpha,\beta]$ with $\alpha < \beta$.\footnote{In this particular case, the interconnection preserves the open-loop equilibrium point $(\bar{\theta},\bar{\omega}) \in \Delta(\Gamma) \times \{0\}$.} Following Theorem \ref{Thm:Absolute}, we examine the loop-transformed system \eqref{Eq:NonlinearSystemLoop}. Let $V(\theta,\omega) = \frac{1}{2}M\omega^2 + bV^2(1-\cos(\theta))$; {\tb this function is strongly convex in a neighbourhood of $(\bar{\theta},0) \in \Theta(\Gamma) \times \{0\}$.} Using \eqref{Eq:Bregman}, a simple computation shows that \eqref{Eq:NonlinearSystemLoop} is quadratically EID with parameters \eqref{Eq:SigmaPrimeSupplyParameters}, where $\varepsilon = (D+\alpha)/(2D + \beta + \alpha )$. {\tb It follows that the closed-loop equilibrium point $(\bar{\theta},0)$ is locally asymptotically stable} for $\alpha > -D$ and $\alpha < \beta < +\infty$. \hfill \oprocend
\end{example}

\section{Equilibrium-Independent Dissipativity for Discrete-Time Control-Affine Systems}
\label{Sec:DTEID}

In this section we consider discrete-time control-affine nonlinear systems {\tb with constant input and throughput matrices}
\begin{equation}\label{Eq:DTNonlinearSystem}
\Sigma:\,\begin{cases}
\begin{aligned}
x_{t+1} &= f(x_t) + Gu_t\\
y_t &= h(x_t) + Ju_t
\end{aligned}
\end{cases}
\end{equation}
{\tb where $t \in \mathbb{Z}_{\geq 0}$ is the time index.} {\tb Similarly to Section \ref{Sec:ControlAffine}, the set of assignable equilibrium points for \eqref{Eq:DTNonlinearSystem} is}
$$
\mathcal{E}_{\Sigma} :=
\begin{cases}
\mathcal{X} & \text{if} \quad m = n\\
\setdef{\bar{x}\in\mathcal{X}}{G^{\perp}(\bar{x}-f(\bar{x})) = \vzeros[n-m]} & \text{if} \quad m < n
\end{cases}
$$ 
with equilibrium-to-input map $\bar{u} = k_u(\bar{x}) = (G^{\sf T}G)^{-1}G^{\sf T}(\bar{x}-f(\bar{x}))$ and equilibrium-to-output map $\bar{y} = k_y(\bar{x}) = h(\bar{x}) + J\bar{u}$. 

\subsection{Review of Discrete-Time Dissipativity}

{\tb  In this subsection we make the additional assumptions that $f(\vzeros[n]) = \vzeros[n]$ and $h(\vzeros[n]) = \vzeros[p]$. Mirroring the definitions from Section \ref{Sec:ClassicDissipativityCT},} the system \eqref{Eq:DTNonlinearSystem} is dissipative with respect to the supply rate \eqref{Eq:SupplyNormal} if there exists a storage function $\map{V}{\mathcal{X}}{\real_{\geq 0}}$ with $V(\vzeros[n]) = 0$ such that
$$
V(x_{t+1}) - V(x_t) \leq {\sf w}(u_t,y_t)\,,
$$
for all $t \in \integer_{\geq 0}$ and {\tb all inputs $u \in \ell_{2e}^{m}[0,\infty)$}. 

While the characterization of continuous-time quadratically dissipative control-affine systems is well understood, the situation for discrete-time control-affine systems is less settled. The cases of lossless and passive systems were studied in \cite{CB-WL:94,WL-CB:95}. Dissipativity with general quadratic supply rates was studied in \cite{ICG-NSS:98} and further generalized to arbitrary supply rates in \cite{EMNL-HSR-RFC:02,EMNL:07}, which is the most general result the author is aware of. All useful known results however are restricted to the situation where the storage function $V(x_{t+1}) = V(f(x_t)+Gu_t)$ evaluated at the next time step is a quadratic function of $u_t$. Under this restriction, the following result is known. 


\smallskip

\begin{lemma}{\bf (Discrete-Time Hill-Moylan Conditions \cite{EMNL:07}):}\label{Lem:DTHillMoylan}
Consider the control-affine system $\Sigma$ in \eqref{Eq:DTNonlinearSystem}. Suppose there exists a twice continuously-differentiable function $\map{V}{\mathcal{X}}{\real_{\geq 0}}$ such that $V(f(x)+Gu)$ is quadratic in $u$. Then $\Sigma$ is dissipative with respect to the supply rate \eqref{Eq:SupplyNormal} with storage function $V(x)$ if and only if there exists an integer $k > 0$ and continuous functions $\map{l}{\mathcal{X}}{\real^k}$, $\map{W}{\mathcal{X}}{\real^{k \times m}}$, such that
\begin{subequations}\label{Eq:DTHillMoylan}
\begin{align}
\label{Eq:DTHillMoylan1}
&\hspace{-5em}\begin{aligned}
V(f(x)) - V(x) &= h(x)^{\sf T}Qh(x)- \|l(x)\|_2^2
\end{aligned}\\
\label{Eq:DTHillMoylan2}
\frac{1}{2}\nabla V(f(x))^{\sf T}G &= h(x)^{\sf T}(QJ+S) - W(x)^{\sf T}l(x)\\
\label{Eq:DTHillMoylan3}
&\hspace{-4em}\begin{aligned}
W(x)^{\sf T}W(x) &= \widehat{R} - \frac{1}{2}G^{\sf T}[\nabla^2 V(f(x))]G\,,
\end{aligned}
\end{align}
\end{subequations}
where $\widehat{R} = R + J^{\sf T}S + S^{\sf T}J + J^{\sf T}QJ$.
%
\end{lemma}

\subsection{Discrete-Time Equlibrium-Independent Dissipativity}

We begin with the key definition.

\smallskip

\begin{definition}{\bf (Discrete-Time EID):}\label{Def:DTEID}
The control-affine system \eqref{Eq:DTNonlinearSystem} is equilibrium-independent dissipative (EID) with  supply rate ${\sf w}(u,y)$ if, for every equilibrium $\bar{x} \in \mathcal{E}_{\Sigma}$, there exists a storage function $\map{V_{\bar{x}}}{\mathcal{X}}{\real_{\geq 0}}$ such that $V_{\bar{x}}(\bar{x}) = 0$ and
\begin{equation}\label{Eq:DTEID}
V_{\bar{x}}(x_{t+1})-V_{\bar{x}}(x_{t}) \leq {\sf w}(u_t-\bar{u},y_t-\bar{y})\,,
\end{equation}
{\tb for all $t \in \mathbb{Z}_{\geq 0}$ and all inputs $u \in \ell_{2e}^{m}[0,\infty)$, where $\bar{u} = k_u(\bar{x})$, $\bar{y} = k_y(\bar{x})$.}
\end{definition}

\medskip

Lemma \ref{Lem:IOMappings} holds for discrete-time systems without changes. To go from dissipativity to equilibrum-independent dissipativity for continuous-time systems in Section \ref{Sec:EID}, we were obliged to (i) strengthen the requirements on the storage function (in the continuous-time case, convexity was assumed), and (ii) replace the first two Hill-Moylan conditions \eqref{Eq:HillMoylan1}--\eqref{Eq:HillMoylan2} with incremental variants. To obtain similar results for discrete-time, we will be obliged to do the same. Here in discrete-time, we strengthen the requirement that $V(f(x)+Gu)$ be quadratic in $u$ to requiring quadratic storage functions $V(x) = x^{\sf T}Px$. 

\medskip

\begin{lemma}{\bf (Conditions for Discrete-Time EID):}\label{Lem:DTHillMoylanIncremental}
Consider the discrete-time control-affine system $\Sigma$ in \eqref{Eq:DTNonlinearSystem}. Let $P = P^{\sf T} \in \real^{n \times n}$ be positive semidefinite, and for $\bar{x} \in \mathcal{E}_{\Sigma}$, let $V_{\bar{x}}(x) := \|x-\bar{x}\|_P^2$. The system $\Sigma$ is EID with respect to the supply rate ${\sf w}(u,y)$ in \eqref{Eq:SupplyNormal} with storage function $V_{\bar{x}}(x)$ if and only if there exists an integer $k > 0$, a matrix $W \in \real^{k \times m}$ and a continuous function {\tb $\map{\ell}{\mathcal{X}\times\mathcal{X}}{\real^k}$} such that
\begin{subequations}\label{Eq:DTHillMoylanEID}
\begin{align}
\label{Eq:DTHillMoylanEID1}
&\begin{aligned}
\|f(x)-&f(\bar{x})\|_P^2 - \|x-\bar{x}\|_P^2\ = - {\tb \|\ell(x,\bar{x})\|_2^2}\\
&\quad +[h(x)-h(\bar{x})]^{\sf T}Q[h(x)-h(\bar{x})]\\
\end{aligned}\\
\label{Eq:DTHillMoylanEID2}
&\begin{aligned}
{\tb [f(x)-f(\bar{x})]^{\sf T}PG} &= {\tb [h(x)-h(\bar{x})]^{\sf T}(QJ+S)}\\
&{\tb \quad - \ell(x,\bar{x})^{\sf T}W }
\end{aligned}\\
\label{Eq:DTHillMoylanEID3}
&
\begin{aligned}
W^{\sf T}W &= \widehat{R} -  G^{\sf T}PG
\end{aligned}
\end{align}
\end{subequations}
where $\widehat{R} = R + J^{\sf T}S + S^{\sf T}J + J^{\sf T}QJ$.
{\tb The function $\ell(x,\bar{x})$ appearing in \eqref{Eq:HillMoylanEID1}--\eqref{Eq:HillMoylanEID2} may always be chosen to have the form 
$$
\ell(x,\bar{x}) = l(x)-l(\bar{x}) + Tq(x,\bar{x})\,,
$$
where $\map{l}{\mathcal{X}}{\real^k}$, the columns of $T \in \real^{k \times r}$ with $r = \dim(\ker(W^{\sf T}))$ form a basis for $\ker(W^{\sf T})$, and $\map{q}{\mathcal{X}\times\mathcal{X}}{\real^r}$ satisfies $q(x,x) = \vzeros[r]$ for all $x \in \mathcal{X}$.}
\end{lemma}

\begin{proof}.
See appendix.
\end{proof}

\medskip
 
Equation \eqref{Eq:DTHillMoylanEID3} {\tb is the third condition from Lemma \ref{Lem:DTHillMoylan}, specialized to a quadratic storage function, while \eqref{Eq:DTHillMoylanEID1}--\eqref{Eq:DTHillMoylanEID2} are incremental variants of the previous conditions \eqref{Eq:DTHillMoylan1}--\eqref{Eq:DTHillMoylan2}.} To interpret the new condition \eqref{Eq:DTHillMoylanEID1}, consider the case where $P \succ \vzeros[]$ and $Q \preceq \vzeros[]$. Then \eqref{Eq:DTHillMoylanEID1} implies that
$$
\|f(x)-f(\bar{x})\|_P^2 \leq \|x-\bar{x}\|_P^2\,.
$$
for all $(x,\bar{x}) \in \mathcal{X} \times \mathcal{E}_{\Sigma}$. If this holds for all $x, \bar{x} \in \mathcal{X}$, then $f$ is non-expansive on $\mathcal{X}$ in the norm $\|\cdot\|_P$. Thus, Lemma \ref{Lem:DTHillMoylanIncremental} replaces the stability-like condition \eqref{Eq:DTHillMoylan1} with the new incremental-stability-like condition \eqref{Eq:DTHillMoylanEID1}. {\tb Internal and feedback stability results for EID systems can be derived in the discrete-time case just as they were in continuous-time in Section \ref{Sec:Stability}; we omit the details, but illustrate the application of these results with an example.}

\medskip


\begin{example}{\bf (Input/Output Gradient Method):}\label{Ex:GradDT}
Consider the unconstrained optimization problem
\begin{equation}\label{Eq:Optimization}
\begin{aligned}
& \underset{x \in \mathbb{R}^n}{\text{minimize}}
& & \phi(x)
\end{aligned}
\end{equation}
where $\map{\phi}{\real^n}{\real}$ is differentiable, $\mu$-strongly convex and $\nabla \phi$ is $L$-Lipschitz, with $0 < \mu \leq L$. Let us define an \emph{input/output gradient method} for \eqref{Eq:Optimization}:
\begin{equation}\label{Eq:GradientClosedLoop}
\begin{aligned}
x_{t+1} &= x_t - \alpha(\nabla \phi(x_t) - v_t)\\
y_{t} &= x_t
\end{aligned}
\end{equation}
where $\alpha > 0$ is the step size and $v \in \ell_{2e}^n[0,\infty)$ is an auxiliary input. {\tb We interpret $v$ as a disturbance to (or error in) the calculated gradient $\nabla \phi(x_t)$, with $v = \vzeros[m]$ recovering the usual gradient method \cite[Sec. 1.2]{DPB:95b}.} {\tb A standard analysis from the optimization literature when $v = \vzeros[m]$ shows that, under the stated assumptions, the gradient method \eqref{Eq:GradientClosedLoop} converges to the unique global minimizer of $\phi$ if $\alpha < \frac{2}{L}$ \cite[Prop. 1.2.3]{DPB:95b}. We will show that this result can be obtained via EID theory.} {\tb To begin, }the system \eqref{Eq:GradientClosedLoop} can be considered as the negative feedback interconnection of the LTI system
\begin{equation}\label{Eq:GradientOpenLoop}
\begin{aligned}
x_{t+1} &= x_t + \alpha u_t\\
y_{t} &= x_t
\end{aligned}
\end{equation}
with the static nonlinearity $\tilde{y}_t = \nabla \phi(\tilde{u}_t)$, i.e., the interconnection $u_t = -\nabla \phi(y_t) + v_t$. {\tb Regarding \eqref{Eq:GradientOpenLoop}, note that $\bar{u} = \vzeros[n]$ is the only possible equilibrium input.} Consider now the function $V(x) = \frac{1}{2\alpha}\|x\|_2^2$, leading to the {\tb candidate EID storage function family} $V_{\bar{x}}(x) = \frac{1}{2\alpha}\|x-\bar{x}\|_2^2$. A simple computation shows that along {\tb trajectories of} \eqref{Eq:GradientOpenLoop},
\begin{equation}\label{Eq:GradientDissipationInequality}
V_{\bar{x}}(x_{t+1}) - V_{\bar{x}}(x_t) = (y_t-\bar{y})^{\sf T}u_t + \frac{\alpha}{2}u_t^{\sf T}u_t\,.
\end{equation}
Therefore, \eqref{Eq:GradientOpenLoop} is EID with supply parameters $(Q,S,R) = (\vzeros[],\frac{1}{2}I_n,\frac{\alpha}{2}I_n)$. 
%

{\tb Since $\phi$ is $\mu$-strongly convex, $\nabla \phi$ is $\mu$-strongly monotone and satisfies the EID inequality \eqref{Eq:EIDNonlinearity} with $(Q,S,R) = (\vzeros[],\frac{1}{2}I_n,-\mu I_n)$. Moreover, since $\nabla \phi$ is both monotone and $L$-Lipschitz, it is also $\frac{1}{L}$-cocoercive \cite[Corollary 18.16]{HB-PC:11}, and therefore $\nabla \phi$ satisfies a second EID inequality with  $(Q,S,R) = (-\frac{1}{L}I_n,\frac{1}{2}I_n,\vzeros[])$.} Taking a convex combination of these two EID inequalities, it follows that for any $\lambda \in [0,1]$, $\nabla\phi$ satisfies the EID inequality \eqref{Eq:EIDNonlinearity} with $(Q,S,R) = (-\lambda\frac{1}{L}I_n,\frac{1}{2}I_n,-(1-\lambda)\mu I_n)$. Applying {\tb (the discrete-time analog of)} Theorem \ref{Thm:EIDFeedback}, {\tb it follows that for any $\lambda \in [0,1]$, the interconnection with input $v_t$ and outputs $(y_t,\tilde{y}_t)$ is EID with supply rate $(Q_{\rm cl},S_{\rm cl},R_{\rm cl})$ given by
$$
Q_{\rm cl} = 
-\begin{bmatrix}
(1-\lambda)\mu & 0\\
0 & \frac{\lambda}{L}-\frac{\alpha}{2}
\end{bmatrix},\, S_{\rm cl} = \begin{bmatrix}
1/2 \\ \alpha/2
\end{bmatrix},\, R_{\rm cl} = \frac{\alpha}{2}\,.
$$
The closed-loop system is internally stable if $Q_{\rm cl} \prec \vzeros[]$}, which is true if and only if $\lambda \in (0,1)$ and {\tb $0 < \alpha < \alpha_{\rm crit}(\lambda) := \frac{2}{L}\lambda$. Maximizing the upper bound $\alpha_{\rm crit}(\lambda)$ over $\lambda \in (0,1)$, we see that $\alpha \in (0,\frac{2}{L})$ is sufficient for stability, which recovers the known step-size result.}

Moving beyond stability to input-output performance, we can examine the equilibrium-independent $\ell_2$-gain for the mapping $v_t \mapsto y_t = x_t$, as a measure of robustness to disturbances. For this mapping, we set $\lambda = 0$ and therefore have EID with respect to the supply rate
{\tb
$$
\begin{aligned}
{\sf w}_{}(v,y) &= -\mu y^{\sf T}y + y^{\sf T}v + \frac{\alpha}{2} v^{\sf T}v\,.
\end{aligned}
$$
}
Applying Lemma \ref{Eq:LemBizzaroL2} (see appendix), we conclude that the I/O mapping $v \mapsto x$ has finite equilibrium-independent $\ell_2$-gain, bounded as
{\tb
\begin{align}\label{Eq:L2Bounds1}
\|\Sigma_{v\,\mapsto\, x}\|_{\ell_2}^2 &\leq \gamma^2 := \frac{1}{\mu^2}\frac{\mu \frac{\alpha}{2}+\frac{1+\sqrt{2\mu \alpha+1}}{4}}{1-\frac{1}{1+\sqrt{2\mu \alpha+1}}}\,.
\end{align}
}
Note that the Lipschitz constant $L$ of $\varphi$ does not enter explicitly\footnote{Of course, in practice $\alpha$ must be selected with some knowledge of $L$ \cite{LL-BR-AP:16}.} into this bound, which depends only on the strong convexity parameter $\mu$ and the step size $\alpha$. {\tb The upper bound is a monotonically increasing function of $\alpha$; small step sizes therefore improve the worst-case I/O performance, but will also lead to slower convergence. Finally, we note that the bound satisfies $\gamma \rightarrow 1/\mu$ as $\alpha \rightarrow 0$ (c.f. \cite[Theorem 4.1]{JWSP:16f}). Therefore, input-output performance is ultimately limited by the modulus of strong convexity of $\phi$}. 
\hfill \oprocend
\end{example}

\medskip

\section{Conclusions}
\label{Section: Conclusions}

This paper has presented a systematic treatment of equilibrium-independent dissipativity for a common class of control-affine nonlinear systems. We have provided a Hill-Moylan-type characterization of EID for both continuous and discrete-time systems, presented some associated internal and feedback stability results, and applied the results to examples in both continuous and discrete time.

Future work will explore applications of these results to the analysis and control of large-scale cyber-physical systems \cite{PJA-BG-VG-MJM-YW-PW-MX-HY-FZ:13}, in particular to applications in power systems. For such applications, extending the present results to differential-algebraic systems would be desirable. Another key direction is to further apply EID and the associated Hill-Moylan conditions developed here to the analysis and design of convex optimization algorithms. In this latter context, an EID-based approach seems particularly well suited due to the presence of monotone nonlinearities, and similar to \cite{LL-BR-AP:16} may provide an intuitive framework for both certifying and improving algorithm performance. Treating EID from a purely input/output point of view is also of interest, {\tb as is developing local versions of the results herein.} 

\section*{Acknowledgements}

The author thanks N. Monshizadeh, C. Nielsen, B. Bamieh, and J. Marden for helpful discussions related to this work.

\renewcommand{\baselinestretch}{1}
\bibliographystyle{IEEEtran}

\bibliography{/Users/jwsimpso/GoogleDrive/JohnSVN/bib/alias,%
/Users/jwsimpso/GoogleDrive/JohnSVN/bib/Main,%
/Users/jwsimpso/GoogleDrive/JohnSVN/bib/JWSP,%
/Users/jwsimpso/GoogleDrive/JohnSVN/bib/New%
}


\begin{IEEEbiography}[{\includegraphics[width=1in,height=1.25in,clip,keepaspectratio]{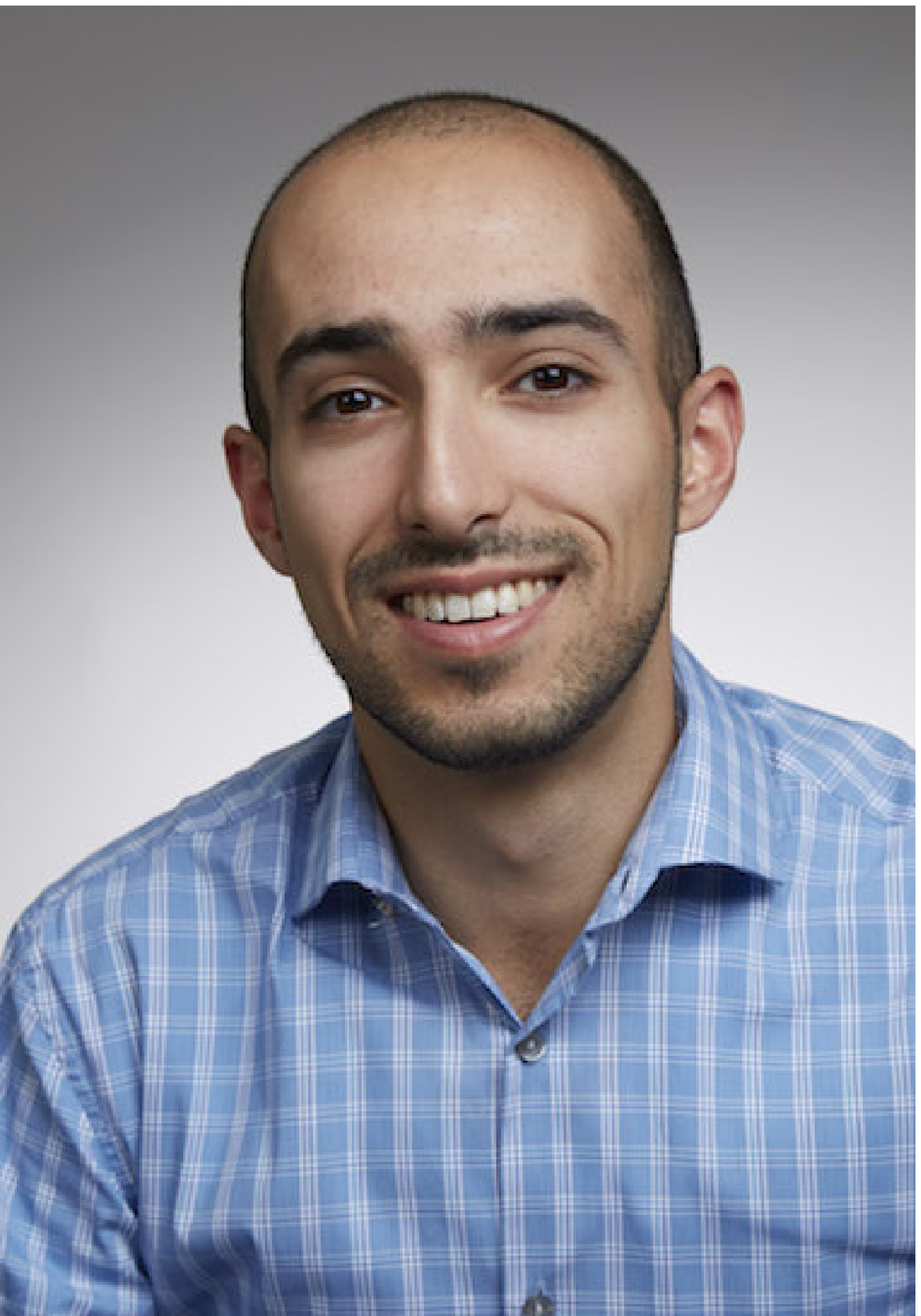}}]{John W. Simpson-Porco} (S'11--M'16) received the B.Sc. degree in engineering physics from Queen's University, Kingston, ON, Canada in 2010, and the Ph.D. degree in mechanical engineering from the University of California at Santa Barbara, Santa Barbara, CA, USA in 2015.

He is currently an Assistant Professor of Electrical and Computer Engineering at the University of Waterloo, Waterloo, ON, Canada. He was previously a visiting scientist with the Automatic Control Laboratory at ETH Z\"{u}rich, Z\"{u}rich, Switzerland. His research focuses on the control and optimization of multi-agent systems and networks, with applications in modernized power grids.

Prof. Simpson-Porco is a recipient of the 2012--2014 IFAC Automatica Prize and the Center for Control, Dynamical Systems and Computation Best Thesis Award and Outstanding Scholar Fellowship.
\end{IEEEbiography}

\appendix[Supporting Lemmas and Proofs]

\begin{lemma}{\bf (Maximal Monotonicity of $\boldsymbol{\mathcal{K}_{\Sigma}}$):}\label{Lem:MaxMonotone}
Suppose that the equilibrium I/O relation $\mathcal{K}_{\Sigma} \subset \real^m \times \real^m$ for a square {\tb continuous-time system \eqref{Eq:NonlinearSystem} (resp. discrete-time system \eqref{Eq:DTNonlinearSystem})} system is monotone. Then $\mathcal{K}_{\Sigma}$ is maximally monotone if
\begin{enumerate}
\item $\mathcal{K}_{\Sigma}$ is $\rho$-cocoercive {\tb with $\rho > 0$}, or
\item $\mathcal{K}_u^{-1} \subseteq \real^m \times \real^n$ is upper hemicontinuous, or
\item $f$ is {\tb a homeomorphism} {\tb (resp. $x \mapsto f(x)-x$ is a homeomorphism)}, or
\item $f$ is the zero map {\tb (resp. the identity map)}.
\end{enumerate}
\end{lemma}

\smallskip

\begin{proof}.
(i): If $\mathcal{K}_{\Sigma}$ is monotone and $\rho$-cocoercive, then it is $\rho^{-1}$-Lipschitz, and is therefore a continuous mapping; continuous monotone mappings are maximally monotone \cite[Corollary 20.25]{HB-PC:11}. (ii): If $\mathcal{K}_u^{-1}$ is upper hemicontinuous, then $\mathcal{K}_{\Sigma} = k_y \circ \mathcal{K}_u^{-1}$ is also upper hemicontinuous, and monotone upper hemicontinuous relations are maximally monotone \cite[Prop. 20.24]{HB-PC:11}. (iii): If $f$ {\tb (resp. $F(x):= f(x)-x$)} is a homeomorphism, then for any $\bar{u} \in \real^m$ there exists $\bar{x} \in \real^n$ satisfying the equilibrium equations $\vzeros[n] = f(\bar{x}) + G\bar{u}$ in continuous-time or $\vzeros[n] = F(\bar{x}) + G\bar{u}$ in discrete-time. {\tb In particular, the solution is a continuous function of $\bar{u}$ and is given by $\bar{x} = k_{u}^{-1}(\bar{u}) = f^{-1}(-G\bar{u})$ (resp. $\bar{x} = F^{-1}(-G\bar{u})$).} It follows that $\mathcal{K}_{\Sigma} = k_y \circ k_u^{-1}$ is a continuous monotone mapping, and is therefore maximally monotone \cite[Cor. 20.25]{HB-PC:11}. (iv): If $f$ is the zero map {\tb (resp. the identity map)}, then $\mathcal{K}_u^{-1} = \setdef{(\bar{u},\bar{x})}{\bar{u} = \vzeros[m]}$, which is upper hemicontinuous, and the result follows from (ii).
\end{proof}

\medskip

\begin{lemma}{\bf (Bregman Divergence Properties):}\label{Lem:Bregman}
Let $\map{V}{\real^n}{\real}$ be differentiable and for $z \in \real^n$ let $V_z(x) := V(x) - V(z) - \nabla V(z)^{\sf T}(x-z)$. If $V$ is (strictly, $\mu$-strongly) convex, then
\begin{enumerate}
\item[(i)] $V_z(x) \geq 0$ (resp. $V_z(x) > 0$, $V_{z}(x) \geq \frac{\mu}{2}\|x-z\|_2^2$) for all $x \neq z$;
\item[(ii)] $x \longmapsto V_z(x)$ is (strictly, strongly) convex;
\end{enumerate}

\end{lemma}

\begin{proof}.
Clearly $V_z(z) = 0$. That $V_z(x) \geq 0$ for $x \neq z$ follows immediately from convexity, since $V_z(x) = V(x) - [V(z) + \nabla V(z)^{\sf T}(x-z)]$ is the difference between $V(x)$ and its linear approximation at $z$, with strict inequality if $V$ is strictly convex. Strong convexity of $V(x)$ is equivalent to 
$$
V(x) - V(z) \geq \nabla V(z)^{\sf T}(x-z) + \frac{\mu}{2}\|x-z\|_2^2
$$
which immediately shows that $V_z(x) \geq \frac{\mu}{2}\|x-z\|_2^2$.  Convexity of $x \mapsto V_z(x)$ follows by directly checking that $V_z(x) - V_z(x^\prime) - \nabla V_{z}(x^\prime)^{\sf T}(x-x^\prime) \geq 0$ for all $x, x^\prime \in \real^n$, with strict inequality when $V$ is strictly convex, and with zero replaced by $\frac{\mu}{2}\|x-x^\prime\|_2^2$ when $V$ is $\mu$-strongly convex.
\end{proof}

%
%
%

\smallskip

\medskip

{


\begin{lemma}{\bf (Intersecting Monotone Relations):}\label{Lem:Intersection}
Let $\mathcal{K}_{\Sigma_1} \subseteq \real^{m} \times \real^{m}$ and $\mathcal{K}_{\Sigma_2} \subseteq \real^{m} \times \real^{m}$ be two maximally monotone relations, each satisfying the dissipation inequality \eqref{Eq:IOMapping} with parameters $(Q_1,\frac{1}{2}I_m,R_1)$ and $(Q_2,\frac{1}{2}I_m,R_2)$, respectively. For any $v_1, v_2 \in \real^m$, the pair of simultaneous inclusions
\begin{equation}\label{Eq:PairOfConclusions}
y_1 \in \mathcal{K}_{\Sigma_1}(v_1-y_2)\,, \qquad y_2 \in \mathcal{K}_{\Sigma_2}(v_2+y_1)\,,
\end{equation}
possess a unique solution if 
$$
R_{2} + Q_{1} \prec \vzeros[] \quad \text{or} \quad R_{1} + Q_{2} \prec \vzeros[]\,.
$$
\end{lemma}

\begin{proof}.
Let $v_1, v_2 \in \real^m$ be arbitrary. Through simple elimination, the pair of inclusions \eqref{Eq:PairOfConclusions} is equivalent to either of the two inclusions
\begin{subequations}
\begin{align}\label{Eq:FIncl}
v_1 &\in F(y_1) := \mathcal{K}_{\Sigma_2}(y_1+v_2) + \mathcal{K}_{\Sigma_1}^{-1}(y_1)\\
\label{Eq:GIncl}
v_2 &\in G(y_2) := \mathcal{K}_{\Sigma_2}^{-1}(y_2) - \mathcal{K}_{\Sigma_1}(-y_2+v_1)
\end{align}
\end{subequations}
where $\mathcal{K}_{\Sigma_1}^{-1} = \setdef{(v,u)}{(u,v)\in\mathcal{K}_{\Sigma_1}}$ is the inverse relation of $\mathcal{K}_{\Sigma_1}$, and similarly for $\mathcal{K}_{\Sigma_2}$. Note that since $\mathcal{K}_{\Sigma_1}$ and $\mathcal{K}_{\Sigma_2}$ are maximally monotone, we have that $Q_{1}, R_{1}, Q_{2}, R_{2} \preceq \vzeros[]$. Consider first the inclusion for $F$. Since $\mathcal{K}_{\Sigma_1}$ and $\mathcal{K}_{\Sigma_2}$ are maximally monotone, it follows that so is $F$ \cite[Prop 20.22]{HB-PC:11}, which satisfies the dissipation inequality \eqref{Eq:IOMapping} with parameters $(R_{2}+Q_{1}, \frac{1}{2}I_m, \vzeros[])$. By (i) then, $F$ is $\mu$-strongly monotone with {\tb $\mu = -\lambda_{\rm max}(R_{2}+Q_{1}) > 0$}, and the inclusion \eqref{Eq:FIncl} possesses a unique solution \cite[Example 22.9]{HB-PC:11}. The second condition follows by applying analogous arguments to the relation $G$.
\end{proof}
}

\begin{pfof}{Lemma \ref{Lem:DTHillMoylanIncremental}} \emph{Sufficiency:} For $(x,\bar{x}) \in \mathcal{X} \times \mathcal{E}_{\Sigma}$, we compute
$$
\begin{aligned}
\Delta V_{\bar{x}} := V_{\bar{x}}&(f(x)+Gu) - V_{\bar{x}}(x)\\ 
&= \|f(x)+Gu-\bar{x}\|_P^2 - \|x-\bar{x}\|_P^2\\
&= \|f(x)-f(\bar{x})+G(u-\bar{u})\|_P^2 - \|x-\bar{x}\|_P^2\\
&= \|f(x)-f(\bar{x})\|_P^2 - \|x-\bar{x}\|_P^2\\
&\quad + 2[f(x)-f(\bar{x})]^{\sf T}PG(u-\bar{u})\\
&\quad + (u-\bar{u})^{\sf T}G^{\sf T}PG(u-\bar{u})\,.
\end{aligned}
$$
Substituting \eqref{Eq:DTHillMoylanEID1} and \eqref{Eq:DTHillMoylanEID3}, we find that
$$
\begin{aligned}
\Delta V_{\bar{x}} &= [h(x)-h(\bar{x})]^{\sf T}Q[h(x)-h(\bar{x})]\\
&\quad - {\tb \|\ell(x,\bar{x})\|_2^2} + 2[f(x)-f(\bar{x})]^{\sf T}PG(u-\bar{u})\\
&\quad + (u-\bar{u})^{\sf T}\widehat{R}(u-\bar{u}) - (u-\bar{u})^{\sf T}W^{\sf T}W(u-\bar{u})\,.
\end{aligned}
$$
Substituting \eqref{Eq:DTHillMoylanEID2}, we further obtain
$$
\begin{aligned}
\Delta V_{\bar{x}} &= [h(x)-h(\bar{x})]^{\sf T}Q[h(x)-h(\bar{x})]\\
&\quad {\tb - \|\ell(x,\bar{x}\|_2^2 -2\ell(x,\bar{x})^{\sf T}W(u-\bar{u})}\\
&\quad +2[h(x)-h(\bar{x})]^{\sf T}(QJ+S)(u-\bar{u})\\
&\quad + (u-\bar{u})^{\sf T}\widehat{R}(u-\bar{u}) - (u-\bar{u})^{\sf T}W^{\sf T}W(u-\bar{u})
\end{aligned}
$$
Adding the nonnegative quantity ${\tb \|\ell(x,\bar{x})+W(u-\bar{u})\|_2^2}$ to the right-hand side of the dissipation equality, after canceling terms we obtain the bound
$$
\begin{aligned}
\Delta V_{\bar{x}} &\leq [h(x)-h(\bar{x})]^{\sf T}Q[h(x)-h(\bar{x})]\\
&\quad + (u-\bar{u})^{\sf T}\widehat{R}(u-\bar{u}) \\ &\quad +2[h(x)-h(\bar{x})]^{\sf T}(QJ+S)(u-\bar{u})
\end{aligned}
$$
Substituting $h(x) = y - Ju$ and collecting terms yields the desired dissipation inequality $\Delta V_{\bar{x}} \leq {\sf w}(u-\bar{u},y-\bar{y})$.

\smallskip

\emph{Necessity:} Suppose that $\Sigma$ is EID with the supply rate \eqref{Eq:SupplyNormal}, i.e., for each $\bar{x} \in \mathcal{E}_{\Sigma}$ it holds that $V_{\bar{x}}(f(x)+Gu) - V_{\bar{x}}(x) \leq {\sf w}(u-\bar{u},y-\bar{y})$. Define the dissipation function
$$
{\sf d}_{\bar{x}}(x,u) := -[V_{\bar{x}}(f(x)+Gu)-V_{\bar{x}}(x)] + {\sf w}(u-\bar{u},y-\bar{y})
$$
which by construction is nonnegative. Using the definition of $V_{\bar{x}}(x)$ and $\bar{x} = f(\bar{x}) + G\bar{u}$, substituting for $y$ and $\bar{y}$, and collecting terms, one finds that
\begin{equation}\label{Eq:DTDissipation}
\begin{aligned}
{\sf d}_{\bar{x}}(x,u) &= \|x-\bar{x}\|_P^2 - \|f(x)-f(\bar{x})+G(u-\bar{u})\|_P^2\\
&+ [h(x)-h(\bar{x})]^{\sf T}Q[h(x)-h(\bar{x})]\\
&+ (u-\bar{u})^{\sf T}\widehat{R}(u-\bar{u})\\
&+ 2[h(x)-h(\bar{x})]^{\sf T}(S+QJ)(u-\bar{u})\\
\end{aligned}
\end{equation}
where $\widehat{R} = R + J^{\sf T}S + S^{\sf T}J + J^{\sf T}QJ$.
{\tb This expression is quadratic in $(u-\bar{u})$, and may be written as
\begin{equation}\label{Eq:DissipationFunctionDT}
\begin{aligned}
{\sf d}_{\bar{x}}(x,u) 
&= \begin{bmatrix}
1 \\ u-\bar{u}
\end{bmatrix}^{\sf T}\underbrace{\begin{bmatrix}
a(x,\bar{x}) & b(x)^{\sf T}-b(\bar{x})^{\sf T}\\
b(x)-b(\bar{x}) & \widehat{R}-G^{\sf T}PG
\end{bmatrix}}_{:= \mathcal{D}(x,\bar{x})}\begin{bmatrix}
1 \\ u-\bar{u}
\end{bmatrix}
\end{aligned}
\end{equation}
where
$$
\begin{aligned}
a(x,\bar{x}) &= \|x-\bar{x}\|_P^2 - \|f(x)-f(\bar{x})\|_P^2\\
&\quad + [h(x)-h(\bar{x})]^{\sf T}Q[h(x)-h(\bar{x})]\\
b(x) &= -f(x)^{\sf T}PG + h(x)^{\sf T}(S+QJ)
\end{aligned}
$$
Arguments similar to those made in the proof of Lemma \ref{Lem:HillMoylanIncremental} show that $\mathcal{D}(x,\bar{x})$ can be factored as
\begin{equation}\label{Eq:MathcalDFactor2}
\mathcal{D}(x,\bar{x}) = \begin{bmatrix}
\ell(x,\bar{x})^{\sf T} \\ W^{\sf T}
\end{bmatrix}\begin{bmatrix}
\ell(x,\bar{x}) & W
\end{bmatrix}
\end{equation}
for an appropriate matrix $W \in \real^{k\times m}$ and function $\map{\ell}{\mathcal{X}\times\mathcal{X}}{\real^k}$. 
Equating the two expressions for $\mathcal{D}(x,\bar{x})$ immediately yields \eqref{Eq:DTHillMoylanEID1}--\eqref{Eq:DTHillMoylanEID3}. The remaining statement follows by arguments identical to those used in the proof of Lemma \ref{Lem:HillMoylanIncremental}.}
%
%
%
%
\end{pfof}

\medskip

\begin{lemma}{\bf (IFP/OSP to Finite $\boldsymbol{\mathscr{L}_2}$-Gain):}\label{Eq:LemBizzaroL2}
If the system $\Sigma$ in \eqref{Eq:NonlinearSystem} is dissipative with respect to the supply rate
$$
{\sf w}(u,y) = -a y^{\sf T}y + y^{\sf T}u + b u^{\sf T}u 
$$
where $a > 0$ and $b \geq 0$, then it is dissipative with respect to the supply rate
$$
{\sf \tilde{w}}(u,y) = - y^{\sf T}y + \gamma^2 u^{\sf T}u 
$$
with
\begin{equation}\label{Eq:gamma-isp-osp}
\gamma^2 = \frac{1}{a^2}\frac{ab+\frac{1+\sqrt{4ab+1}}{4}}{1-\frac{1}{1+\sqrt{4ab+1}}}
\end{equation}
\begin{proof}.
Let $\delta > 1/(2a)$, then
$$
\begin{aligned}
{\sf w}(u,y) &= -a y^{\sf T}y - \underbrace{\frac{1}{2\delta}(y-\delta u)^{\sf T}(y-\delta u)}_{\geq 0}\\
&\quad + b u^{\sf T}u + \frac{\delta}{2}u^{\sf T}u + \frac{1}{2\delta}y^{\sf T}y\\
&\leq -\left(a-\frac{1}{2\delta}\right)y^{\sf T}y + \left(b + \frac{\delta}{2}\right)u^{\sf T}u
\end{aligned}
$$
After rescaling by $a - \frac{1}{2\delta} > 0$, this is equivalent to dissipativity with respect to the supply rate
$$
{\sf \bar{w}}(u,y) = -y^{\sf T}y + \Gamma(\delta) u^{\sf T}u
$$
with
$$
\Gamma(\delta) = \frac{b + \frac{\delta}{2}}{a-\frac{1}{2\delta}}\,.
$$
The function $\Gamma(\delta)$ is {\tb strictly convex on its domain $(\frac{1}{2a},\infty)$}, and achieves its {\tb global} minimum of $\gamma^2$ at $\delta^\star = (\sqrt{4ab+1}+1)/(2a)$, where $\gamma$ is as in \eqref{Eq:gamma-isp-osp}.\end{proof}
\end{lemma}

{\tb
\begin{lemma}\label{Lem:Orthogonal}
Let $\widehat{R} \in \real^{m \times m}$ be positive semidefinite, and let $\map{W}{\real^N}{\real^{k \times m}}$. Then $W(z)^{\sf T}W(z) = \widehat{R}$ for all $z \in \real^N$ if and only if there exists an orthogonal matrix $\mathcal{O}(z) \in \real^{k \times k}$ and a constant matrix $W^\prime \in \real^{k \times m}$ such that $(W^\prime)^{\sf T}W^\prime = \widehat{R}$ and $W(z) = \mathcal{O}(z)W^\prime$ for all $z \in \real^N$. 
\end{lemma}
\begin{proof}.
That the existence of such quantities is sufficient for $W(z)^{\sf T}W(z) = \widehat{R}$ is straightforward. To show necessity, first note (trivially) that $W(z)^{\sf T}W(z)$ and $\widehat{R}$ commute. It follows by applying \cite[2.6.P11]{RAH-CRJ:12} point-wise that there exist orthogonal matrices $U(z) \in \real^{k \times k}$ and $V \in \real^{m \times m}$ and diagonal matrices $\Sigma \in \real^{k \times m}$ and $\Lambda \in \real^{m \times m}$ such that
$\widehat{R} = V\Lambda V^{\sf T}$ and $W(z) = U(z)\Sigma V^{\sf T}$; the result follows then with $\mathcal{O}(z) = U(z)$ and $W^\prime = \Sigma V^{\sf T}$.
%
%
%
\end{proof}
}

{\tb
\begin{lemma}\label{Lem:DifferenceFunction}
Let $\map{f}{\real^n \times \real^n}{\real^m}$. The following two statements are equivalent:
\begin{enumerate}[(i)]
\item $f(x_1,x_2) + f(x_2,x_3) + f(x_3,x_1) = \vzeros[m]$ for all $x_1,x_2,x_3 \in \real^n$
\item there exists a function $\map{g}{\real^n}{\real^m}$ such that $f(x_1,x_2) = g(x_1) - g(x_2)$ for all $x_1,x_2 \in \real^n$.
\end{enumerate}
\end{lemma}
\begin{proof}.
The implication (ii) $\Rightarrow$ (i) is immediate. To show that (i) $\Rightarrow$ (ii), first set $x_1 = x_2 = x_3$ to find that $f(x_1,x_1) = \vzeros[m]$. Similarly, set $x_1 = x_3$ to find that
$$
f(x_1,x_2) + f(x_2,x_1) + \underbrace{f(x_1,x_1)}_{=\vzeros[m]} = \vzeros[m]
$$
which shows that $f(x_1,x_2) = -f(x_2,x_1)$. Finally, set $g(x) = f(x,\vzeros[n])$ and set $x_3 = \vzeros[n]$ in (i) to find that
$$
f(x_1,x_2) = -f(x_2,\vzeros[n]) - f(\vzeros[n],x_1) = -g(x_2) + g(x_1)\,,
$$
which shows the result.
\end{proof}
}

\end{document}